\documentclass[11pt]{amsart}

\usepackage[a4paper,inner=3.6cm,outer=3.6cm,top=3.7cm,bottom=3.7cm]{geometry}
\usepackage{amsfonts,mathrsfs}
\usepackage{amssymb}
\usepackage{amsmath}
\usepackage[]{graphicx}
\usepackage[all,cmtip,2cell]{xy}
\usepackage{comment}
\usepackage{array}
\usepackage{enumitem}
\usepackage{mathtools}
\usepackage{multirow}
\usepackage{rotating,tabularx}
\usepackage{courier}
\usepackage{comment}

\setlist[itemize]{leftmargin=*}
\setlist[enumerate]{leftmargin=*}

\usepackage[usenames,dvipsnames]{xcolor}

\def\co{\colon\thinspace}
\newcommand{\C}{\mathbb{C}}
\newcommand{\Z}{\mathbb{Z}}

\newcommand{\R}{\mathbb{R}}
\newcommand{\del}{\partial}

\newcommand{\M}{\mathcal{M}}
\newcommand{\cL}{\mathcal{L}}

\newcolumntype{C}[1]{>{\centering\arraybackslash$}p{#1}<{$}}

\newcommand{\pt}{\mathrm{pt}}

\newcommand{\BS}{\mathcal{BS}}

\renewcommand{\L}{\mathcal{L}}
\newcommand{\xx}{\mathbf{x}}

\makeatletter
\newcommand{\dashover}[2][\mathop]{#1{\mathpalette\df@over{{\dashfill}{#2}}}}
\newcommand{\fillover}[2][\mathop]{#1{\mathpalette\df@over{{\solidfill}{#2}}}}
\newcommand{\df@over}[2]{\df@@over#1#2}
\newcommand\df@@over[3]{%
	\vbox{
		\offinterlineskip
		\ialign{##\cr
			#2{#1}\cr
			\noalign{\kern1pt}
			$\m@th#1#3$\cr
		}
	}%
}
\newcommand{\dashfill}[1]{%
	\kern-.5pt
	\xleaders\hbox{\kern.5pt\vrule height.4pt width \dash@width{#1}\kern.5pt}\hfill
	\kern-.5pt
}
\newcommand{\dash@width}[1]{%
	\ifx#1\displaystyle
	2pt
	\else
	\ifx#1\textstyle
	1.5pt
	\else
	\ifx#1\scriptstyle
	1.25pt
	\else
	\ifx#1\scriptscriptstyle
	1pt
	\fi
	\fi
	\fi
	\fi
}
\newcommand{\solidfill}[1]{\leaders\hrule\hfill}
\makeatother

\newcommand{\bM}{\overline{\mathcal{M}}}

\newcommand{\cR}{\mathcal{R}}

\newcommand{\en}{\bullet}

\newtheorem{theorem}{Theorem}[section]
\newtheorem{proposition}[theorem]{Proposition}
\newtheorem{lemma}[theorem]{Lemma}
\newtheorem{corollary}[theorem]{Corollary}
\newtheorem{conjecture}[theorem]{Conjecture}
\theoremstyle{definition}
\newtheorem{definition}[theorem]{Definition}

\theoremstyle{remark}
\newtheorem{remark}{Remark}[section]
\newtheorem{example}[remark]{Example}

\numberwithin{equation}{section}

\begin{document}

	\title[String topology and
	LG periods]{String topology with gravitational descendants,\\ and
		periods of Landau-Ginzburg potentials}
	
	\author{Dmitry Tonkonog}
	\let\thefootnote\relax\footnote{This work was partially supported by the Simons
		Foundation grant \#385573, Simons Collaboration on Homological Mirror
		Symmetry.}
	\address{University of California, Berkeley}
	
	\begin{abstract}
	This paper introduces new operations on the string topology of a smooth manifold: gravitational descendants of its cotangent bundle, which are augmentations of the Chas-Sullivan $L_\infty$ algebra structure of the loop space. The definition extends to Liouville domains. Descendants of the $n$-torus are computed.
	
	To a monotone Lagrangian torus in a symplectic manifold, one associates a Laurent polynomial called the Landau-Ginzburg potential, by counting holomorphic disks. This paper proves the following mirror symmetry prediction: the constant terms of the powers of an LG~potential are equal to descendant Gromov-Witten invariants of the ambient manifold.
	\end{abstract}
	\maketitle
	
	\section{Overview}

\subsection{Landau-Ginzburg potential}
Given a smooth Fano variety considered as a symplectic manifold, one question that can be asked about it is what monotone Lagrangian tori does it contain. This paper explains how Lagrangian tori shed light on the enumerative geometry of the Fano variety, and vice versa.

The interest in monotone Lagrangian tori  generally stems from mirror symmetry;  concretely, recent constructions of such reveal exciting algebraic and combinatorial patterns which seem to correctly capture the structure of cluster charts of the mirror variety. 
The constructions include those of Lagrangian tori in $\C P^2$ (indexed by Markov triples, i.e.~solutions of the Markov equation $a^2+b^2+c^2=3abc$) and  del Pezzo surfaces by Vianna \cite{Vi13,Vi14,Vi16}; tori in $\R^6$ by Auroux \cite{Au15}; and higher-dimensional mutations of Lagrangian tori in toric Fano varieties by Pascaleff and the author \cite{PT17}.

One is naturally interested in monotone Lagrangian tori up to Hamiltonian isotopy. Given such a torus $L\subset X$ with a fixed basis of $H_1(L;\Z)\cong \Z^n$, one defines the \emph{Landau-Ginzburg potential} of $L$  to be the Laurent polynomial
$$
W_L\in\C[x_1^{\pm 1},\ldots,x_n^{\pm 1 }],
$$
where $x_1,\ldots,x_n$ are formal variables associated with the basis of $H_1(L;\Z)$, in the following way. The LG potential counts $J$-holomorphic Maslov index~2 disks $(D,\del D)\subset(X,L)$ whose boundary passes through a specified point on $L$, for some fixed compatible almost complex structure $J$. There is a finite number of such disks;  the potential also retains the information about their boundary homology classes. Namely, to a disk $D$  with $[\del D]=(v^1,\ldots,v^n)\in\Z^n=H_1(L;\Z)$ one associates the monomial $\pm x_1^{v^1}\ldots x_n^{v^n}$, and $W_L$ is obtained as the sum of these monomials over all holomorphic disks as above; see e.g.~\cite{CO06,Au07,FO310}. The  signs arise from the orientation on the moduli space of holomorphic disks.

This definition is not specific to tori but the requirement that $L\subset X$ be \emph{monotone} is crucial: it guarantees that $W_L$ is invariant of the choice of~$J$ and Hamiltonian isotopies of $L$.
Recall that a Lagrangian submanifold $L\subset X$ is called monotone if the following two maps:
$$
\omega\co H_2(X,L;\Z)\to \R,\quad \mu\co H_2(X,L;\Z)\to \Z,
$$
the symplectic area and the Maslov index, are positively proportional to each other. This in particular implies that $\omega\in H^2(X;\Z)$ and $c_1(X)\in H^2(X;\Z)$ are positively proportional, that is, $X$ is a monotone symplectic manifold. Fano varieties constitute the main class of monotone symplectic manifolds.

\begin{example}
The product torus in $\R^4=\C^2(z_1,z_2)$ given by $\{|z_1|=r_1,\ |z_2|=r_2\}$ is monotone if and only if $r_1=r_2$. Its potential equals $W=x_1+x_2$.
\end{example}

\begin{example}
	\label{ex:toric_pot}
Let $X$ be a (smooth) compact toric Fano variety, $\Delta$ its moment polytope and $\pi\co X\to \Delta$ the moment map. The preimages of interior points of $\Delta$ are Lagrangian tori in $X$, and there is a unique point in $\Delta$ (the origin using the common normalisation of $\pi$) whose preimage is monotone.  The potential of this monotone torus has been computed by Cho and Oh \cite{CO06}: it is equal to the standard Hori-Vafa potential \cite{HV00}, cf.~\cite{Bat94}. For example, for $\C P^2$ this is $x_1+x_2+{x_1}^{-1}x_2^{-1}$.
\end{example}

\subsection{Quantum periods theorem}
Let $W\in \C[x_1^{\pm 1},\ldots,x_n^{\pm 1}]$ be a Laurent polynomial. Its \emph{$d$th period} 
$
\phi_d(W)\in\C
$
is the constant term of the $d$th power $W^d$. 

\begin{example}
One has $\phi_3(x_1+x_2+x_1^{-1}x_2^{-1})=6$.
\end{example}

Let $X$ be a compact monotone symplectic manifold.
The \emph{$d$th quantum period of $X$} is the number  $$\langle\psi_{d-2}\,\pt\rangle_{X,d}\in\mathbb{Q}.$$ It is the gravitational descendant one-pointed Gromov-Witten invariant of the point class $\pt \in H^{2n}(X)$, defined by the formula
\begin{equation}
\label{eq:psi_class}
\langle\psi_{d-2}\,\pt\rangle_{d}\coloneqq \int_{[\bM_1(X,d)]}c_1(\L)^{d-2}\cup ev^*([pt]).
\end{equation}
See Section~\ref{sec:cm} or  \cite{CoKa99,Giv94,Giv96} for a reminder.

\begin{theorem}[Quantum periods theorem]
	\label{th:lg_per}
	Let $X$ be a closed monotone symplectic $2n$-manifold and $L\subset X$ a monotone Lagrangian torus with LG potential $W_L\in\C[x_1^{\pm 1},\ldots, x_n^{\pm 1}]$. Then for all $d\ge 2$,
	$$
	 \phi_d(W_L)=d!\, \langle\psi_{d-2}\,\pt\rangle_{X,d}.
	$$
\end{theorem}

\begin{corollary}
	\label{cor:per_eq}
	Let $X$ be a monotone symplectic manifold. For each $d$, the $d$th periods of the LG~potentials of all monotone Lagrangian tori in $X$ are equal to each other.\qed
\end{corollary}

\subsection{Tangential and descendant invariants}
It is important to distinguish descendant invariants from related, yet different ones: relative Gromov-Witten invariants, or Gromov-Witten invariants with tangency conditions.
Specifically, consider the following invariant. Fix a point $y\in X$, a compatible almost complex structure on $X$ which is integrable in a neighbourhood $U$ of $y$, and a germ $Y\subset U$ of a $J$-complex hypersurface passing through $y$. Following Cieliebak and Mohnke \cite{CM14}, one defines the \emph{Gromov-Witten invariant with a tangency}
$$
\langle\tau_{d-2}\,\pt\rangle_{X,d}\in\Z
$$
to be the count of $J$-holomorphic Chern number~$d$ spheres in $X$ which pass through $y\in X$ and have intersection multiplicity $d-1$ with $Y$ at that point. See Section~\ref{sec:cm} for more details. 

Importantly, the relation between $\langle \psi_{d-2}\, \pt\rangle_{X,d}$ and $\langle \tau_{d-2}\, \pt\rangle_{X,d}$ is complicated. GW invariants with a tangency have been recently revisited by Siegel \cite{Sie19}; among the interesting computations made in that paper are the invariants of $\C P^2$, see \cite[Section~5.5]{Sie19}. Section~\ref{sec:cm} expands on the relation and the difference between the $\psi$- and $\tau$-versions of the invariant.

While $\langle\tau_{d-2}\,\pt\rangle_{X,d}$ is defined by counting the elements of a zero-dimensional moduli space, the definition of $\langle\psi_{d-2}\,\pt\rangle_{X,d}$ does not have that form.
For the proof of Theorem~\ref{th:lg_per}, it is convenient to have an alternative reformulation of $\langle\psi_{d-2}\,\pt\rangle_{X,d}$ as a zero-dimensional counting problem. 
Section~\ref{sec:cm} introduces such a problem, called enumerative descendants, and $$\langle\psi_{d-2}\,\pt\rangle_{X,d}^\en$$ stands for corresponding count.
The following comparison is proved in Section~\ref{sec:cm}.

\begin{lemma}
	\label{lem:desc}
	For any monotone symplectic manifold $X$ and any $d\ge 2$
	it holds that $$\langle\psi_{d-2}\,\pt\rangle_{X,d}^\en=(d-2)!\cdot  \langle\psi_{d-2}\,\pt\rangle_{X,d}.$$
\end{lemma}

One can try to compare  $\langle \psi_{d-2}\, \pt\rangle_{X,d}$ with $\langle \tau_{d-2}\, \pt\rangle_{X,d}$ as follows, see Section~\ref{sec:cm} for details. On the moduli space of holomorphic spheres in $X$ passing through $y\in X$, consider the higher jets of these curves at $y$ in the direction normal to a local divisor $Y\ni y$. (These jets are defined whenever the lower-order ones vanish). On one hand, the vanishing loci of the jets cut out moduli spaces of curves with increasing tangencies to $Y$, and on the other hand, these jets are sections of the powers of $\cL$, the line bundle appearing in the definition of the $\psi$-invariants. 

Section~\ref{sec:cm} reminds that
this observation does not directly translate to an identity between the $\tau$- and $\psi$-invariants because of bubbled configurations with a {\it constant} holomorphic bubble inheriting the tangency condition. To remedy this, Section~\ref{sec:cm} introduces a modified problem where the curves carry a Hamiltonian perturbation near the tangency point to $Y$, with Hamiltonian vector field transverse to $Y$. These moduli spaces can no longer develop constant bubbles at the tangency point (Lemma~\ref{lem:no_const_bub}). They are used to define the enumerative descendants $\langle\psi_{d-2}\,\pt\rangle_{X,d}^\en$.

Later, the proof of Theorem~\ref{th:lg_per} actually goes by showing that
$$
\phi_d(W_L)=d(d-1)\cdot \langle\psi_{d-2}\,\pt\rangle_{X,d}^\en.
$$

The background placing Theorem~\ref{th:lg_per} and Corollary~\ref{cor:per_eq} into context will be surveyed soon, along with their applications. 
The following generalisation of Theorem~\ref{th:lg_per} is  proved as well.

\begin{theorem}
	[=Theorem~\ref{th:gw_bs}]
	\label{th:gw_s_intro}
	Let $X$ be a closed monotone symplectic manifold, and $M\subset X$ a monotone Liouville subdomain  admitting a non-negatively graded Floer complex, and a suitable Donaldson divisor in its complement.
	Then
	$$
	\langle\underbrace{ \BS|\ldots|\BS}_d
	\rangle_M=d!\,\langle\psi_{d-2}\,\pt\rangle^\en_{X,d}.
	$$
\end{theorem}

Here is a quick explanation of the statement, without going into the forthcoming details. A \emph{monotone}  Liouville domain $M\subset X$ is a notion generalising the Weinstein neighbourhood of a monotone Lagrangian submanifold; the definition is  simple but  does not seem to have  appeared in the literature before. The \emph{Borman-Sheridan class} $\BS\in SH^0_{S^1,+}(M)$ accordingly generalises the Landau-Ginzburg potential. Here $SH^0_{S^1,+}(M)$  is the positive equivariant symplectic cohomology, graded so that the Viterbo isomorphism reads $SH^*_{S^1,+}(T^*L)\cong H_{n-*-1}(\L L/S^1,L)$. For example, $SH^0_{S^1,+}(T^*T^n)\cong\C[x_1^{\pm1},\ldots,x_n^{\pm 1}]$ as vector spaces (recall that the left hand side is not a ring).

\begin{remark}
	The Borman-Sheridan class has a non-equivariant version, but the statement Theorem~\ref{th:gw_s_intro} needs the equivariant one. 
\end{remark}

One says that $M$ has \emph{non-negatively graded Floer complex} if $\del M$ admits a contact structure such that the Floer complex $CF^*_{S^1,+}(M)$ computing the symplectic cohomology is concentrated in non-negative degrees up to an  arbitrarily high action truncation: for example, cotangent bundles of manifolds admitting  a metric with non-positive sectional curvature (like the $n$-torus) have this property. 

Finally, the brackets from the left hand side of the identity from Theorem~\ref{th:gw_s_intro} are gravitational descendant operations 
$$
\langle\, \cdot\, | \ldots |\, \cdot \, \rangle_M\co SH^0_{S^1,+}(M)^{\otimes d}\to \C
$$
introduced in this paper.
These operations are only defined at the cohomology level   when $M$ admits  non-negatively graded Floer complex. 
The main step towards the proof of Theorem~\ref{th:lg_per} is to compute descendants of $M=T^*T^n$; the statement is found in
Theorem~\ref{th:m_comput} and the proof occupies Section~\ref{sec:torus}.

For a general Liouville domain, gravitational descendants are chain-level operations; they are quickly overviewed next.

\subsection{Gravitational string topology}
Let $M$ be a Liouville domain (see the references in Section~\ref{sec:bs});  the cotangent bundle $M=T^*L$ of a smooth manifold $L$ is  already a very interesting case for this discussion.
 The Floer chain complex $CF^*(M)$ computing symplectic cohomology has a natural structure of an $L_\infty$ algebra. 
 
 The {\it positive equivariant} Floer complex $CF^*_{S^1,+}(M)$ is more relevant for the present paper. It also has an $L_\infty$ algebra structure: this means that there is a sequence of operations
 $$
 l^k\co CF^*_{S^1,+}(M)^{\otimes k}\to CF^{*+3-2k}_{S^1,+}(M)
 $$
 satisfying the $L_\infty$ relations appearing in Section~\ref{sec:grav}. 
 
 For each $m,k\ge 1$,  this paper introduces a \emph{gravitational descendant} operation which is a degree $2-2k$ map
 $$
 \psi^k_{m-1}\co CF^*_{S^1,+}(M)^{\otimes k}\to \C[-2m]
 $$
 where $\C[-2m]$ is a copy of $\C$ of grading $-2m$.
 In other words the operation $\psi^k_{m-1}$ vanishes unless the degrees of the inputs sum to $2k-2-2m$, and otherwise returns a  number. This is the point where using the equivariant Floer complex is important: this augmentation does not have a non-equivariant analogue.
 
 \begin{figure}[h]
 	\includegraphics[]{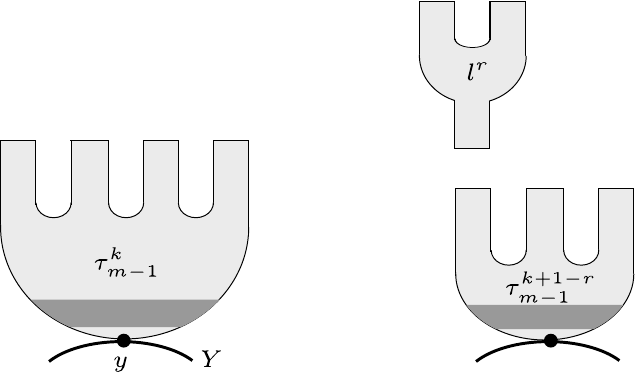}
 	\caption{Left: a gravitational descendant operation. The shaded region contains a Hamiltonian perturbation. Right: the bubbling showing part of the $L_\infty$ augmentation relation ($k=4$, $r=2$).}
 	\label{fig:t_l_inf}
 \end{figure}
 
 Here is a rough definition of $\psi^k_{m-1}$. These operations count holomorphic maps $\C P^1\setminus\{z_1,\ldots,z_k\}\to M$ with an additional marked point $z_0\in \C P^1$, asymptotic to the given input orbits  $x_i\in CF^*_{S^1,+}(M)$ at the punctures $z_i$, and passing through a specified point $y\in M$ at $z_0$ with intersection multiplicity $m$ with a fixed germ $Y$ of a complex hypersurface. 
 The last condition is analogous to the one from the  definition of $\langle\psi_{d-2}\,\pt\rangle_d^\en$; see Figure~\ref{fig:t_l_inf}, left. Similarly to the definition of $\langle\psi_{d-2}\,\pt\rangle_d^\en$, the curves defining $\psi^k_{m-1}$ carry a Hamiltonian perturbation near the tangency point. See Section~\ref{sec:grav} for the details.

 The stated degree of $\psi^k_{m-1}$ means that one is counting the moduli spaces which are zero-dimensional. Taking the boundary of 1-dimensional moduli spaces defined in the same way, see Figure~\ref{fig:t_l_inf}, right, one obtains the following identities between gravitational descendants and the $L_\infty$ structure:
 $$
 \sum_{\stackrel{1\le r\le k,}{\sigma\in S_k}}(-1)^\maltese\, \tfrac{1}{r!(k-r)!}\, \psi^{k+1-r}_{m-1}(l^{r}(x_{\sigma_1},\ldots x_{\sigma_{r}}),x_{\sigma_{r+1}},\ldots, x_{\sigma_{k}})=0.
 $$
These identities may be rephrased to say that for each $m\ge 1$, the collection of maps $\psi_{m-1}=\{\psi_{m-1}^k\}_{k\ge 1}$ is an \emph{$L_\infty$ morphism} from the $L_\infty$ algebra $CF^*_{S^1,+}(M)$ to the shifted one-dimensional vector space $\C[-2m]$ considered as the trivial $L_\infty$ algebra.
  Yet another equivalent formulation is that 
  $\psi_{m-1}=\{\psi_{m-1}^k\}_{k\ge 1}$ is  a shifted \emph{augmentation} of the $L_\infty$ algebra $CF^*_{S^1,+}(M)$.

\begin{remark}
Siegel \cite{Sie19} studied a version of these operations without using a Hamiltonian perturbation; call these operations $\tau^{k}_{m-1}$. They are different from the ones defined in this paper: generally, $\tau^{k}_{m-1}\neq \psi^{k}_{m-1}$. However, both give (different) $L_\infty$ augmentations of the same equivariant Floer complex.  The reason for the difference between $\tau^{k}_{m-1}$ and $\psi^{k}_{m-1}$ is similar to the closed-string case. Briefly, when one turns off the Hamiltonian perturbation used in the definition of the ``$\psi$-curves'' to zero, these curves may converge to nodal unions of the ``$\tau$-curves'' rather than a just a single ``$\tau$-curve''.
\end{remark}
  
  Let $L$ be a smooth orientable spin manifold, and $\L L$ its free loop space.
  In view of the Viterbo isomorphism $SH^*_{S^1,+}(T^*L)\cong H_{n-*-1}(\L L/S^1,L)$ and the results extending it, the Floer complex  $CF^*_{S^1,+}(T^*L)$ is understood as a model for the \emph{string topology} of $L$, that is, for the space of chains $C_*(\L L/S^1,L)$ on the unparametrised loop space of $X$. This space carries a wealth of operations (see the references in Section~\ref{sec:grav}) which go under the general name of {\it string topology}. For example, the  $L_\infty$ algebra on the Floer complex is conjecturally quasi-isomorphic to the Chas-Sullivan $L_\infty$ algebra structure on $C_*(\L L/S^1,L)$, compare \cite{Irie18}.
  
  From this point of view, gravitational descendants $\psi_{m-1}^k$ of $T^*L$ are new string topology operations. Curiously, they have no obvious explicit interpretation as operations defined geometrically on the chains on the loop space; indeed, it is not clear how to reinterpret the tangency condition for the holomorphic curves in terms of chains on $\L L$.
  It would be  interesting to obtain such a description, which might involve looking at the strata of  intersections  of loop cycles where the intersections happen less generically than  transversally. It might also be possible to formulate these operations within the open conformal field theory framework for string topology \cite{Sul07,God07}.

\subsection{Periods and wall-crossing}  
 For a Laurent polynomial $W$ in $n$ variables,  a version of the Cauchy formula reads
 $$
 \phi_d(W)=\frac{1}{(2\pi i)^n}
 \cdot \int_{|x_1|=\ldots=|x_n|=1}W^d(x_1,\ldots,x_n)\cdot \frac{dx_1}{x_1}\wedge\ldots\wedge \frac{dx_n}{x_n}
 $$
 justifying the name \emph{period} for $\phi_d(W)$. Here $x_i$ are considered as complex variables.
 This formula makes it obvious that the periods of $W$ do not change if $W$ is modified  by a birational substitution of variables
 \begin{equation}
 \label{eq:wall_cross}
 \begin{array}{lll}
 x_1&\mapsto& x_1,\\
 \ldots\\
 x_{n-1}&\mapsto& x_{n-1},\\
 x_n&\mapsto& x_n\cdot f(x_1,\ldots,x_{n-1}),
 \end{array}
 \end{equation}
  where $f$ is a holomorphic function. Provided such a substitution transforms $W$ into another Laurent polynomial, this is called a \emph{mutation} of $W$, or a \emph{wall-crossing} of $W$, see e.g.~\cite{Prz07,Katzarkov2011,GU12,ACGK12,CMG13,PT17}. In particular the lemma below is easily seen from the Cauchy formula.
  
  \begin{lemma}
  	\label{lem:per_cauchy}
  The periods of a Laurent polynomial remain unchanged under mutation.\qed
  \end{lemma}
  
 There exist geometric modifications of monotone Lagrangian tori in monotone symplectic manifolds which are also called \emph{mutations}; they have been studied by Vianna \cite{Vi13} and Shende, Treumann and Williams \cite{STW15} in dimension four, and by Pascaleff and the author \cite{PT17} for higher-dimensional toric Fanos. The wall-crossing formula for such mutations was proved in \cite{PT17}; it states that the Landau-Ginzburg potential of a monotone torus changes by a specific wall-crossing (\ref{eq:wall_cross}) when a Lagrangian torus is mutated geometrically. (In these known  wall-crossing formulas, the functions $f$  from (\ref{eq:wall_cross}) are of the form $1+x_1+\ldots+x_k$ for appropriately chosen $H_1$-bases for the tori.)
 Therefore the known geometric mutations of Lagrangian tori do not change the periods of their Landau-Ginzburg potentials, in agreement with Theorem~\ref{th:lg_per}. 
  
 \subsection{Mirror symmetry}
 Given a Laurent polynomial $W$,
 one introduces
 the \emph{classical period of $W$}, or its \emph{non-regularised constant term series}, by
 $$
 \pi_W=\sum_{d\ge 0}\frac 1 {d!}\,\phi_d(W)\cdot t^d.
 $$

 Let $X$ be a monotone symplectic manifold.
 Givental's $J$-series of $X$ is a generating function for its Gromov-Witten invariants with gravitational descendants. Restricting to descendant Gromov-Witten invariants of the point class $\pt\in H_0(X)$, one defines the
 \emph{quantum period of $X$}, or the \emph{fundamental term of Givental's $J$-series} to be the following power series in the formal variable $t$:
 $$
 G_X=1+\sum_{d\ge 2}\langle\psi_{d-2}\,\pt\rangle_{X,d}\cdot t^d.
 $$
 One says that $W$ is a \emph{mirror dual to $X$}, see e.g.~\cite[Definition~4.9]{CCGGK14}
 if $$\pi_W=G_X.$$ Surveys of this topic are found in \cite{CCGGK14,CCGK16}; it has also been studied in e.g.~\cite{Prz07,Katzarkov2011,Prz13,Prz17,GalThesis} where a mirror dual potential is called a \emph{very weak Landau-Ginzburg model}. In these terms,
 Theorem~\ref{th:lg_per} can be reformulated as follows.
 
 \begin{corollary}
 Let $L\subset X$ be a monotone Lagrangian torus. Its LG~potential is mirror dual to $X$, that is, $\pi_{W_L}=G_X$.
 \end{corollary}

 This corollary can be seen as a vast generalisation of the following theorem of Givental. It follows from his proof of mirror symmetry for toric varieties, specifically from his computation of their $J$-function \cite{Giv96}; see~\cite[Corollary~C.2]{CCGK16}, \cite{CCGGK14} for details, and also recall Example~\ref{ex:toric_pot}.
  (Sample computations extracting quantum periods from the $J$-function are also found in \cite[Example~10.1.3.1]{CoKa99} and \cite[Example~5.13]{Guest08}.)
 
 \begin{theorem}
 Let $X$ be a toric Fano variety and $W$ its standard toric potential. Then $W$ is mirror dual to $X$.\qed
 \end{theorem}

 \begin{example}
 	The Laurent polynomial
 	$$
 	W=x_1+\ldots+x_{n-1}+x_1^{-1}\ldots x_{n-1}^{-1}
 	$$
 	is mirror dual to $\C P^{n-1}$. The periods of $W$ are easily computed: for each $r\ge 1$
 	$$
 	\phi_{nr}(W)=
 	\binom {nr} {n,n,\ldots,n}
 	=\frac{(nr)!}{(n!)^r}
 	$$
 	where the middle term is a multinomial coefficient.
 	The quantum periods of $\C P^{n-1}$ are
 	$$
 	\langle \psi_{nr-2}\,\pt\rangle_{\C P^{n-1},\,nr}=\frac{1}{(n!)^r}.
 	$$
 \end{example}

 The definition of a Laurent polynomial dual to a Fano variety has the following motivation.
 Under suitable assumptions, the series $G_X$ and $\pi_W$ are solutions to the following  differential equations, respectively: the flatness equation with respect to the Dubrovin connection  and the Picard-Fuchs equation.
 The classical  mirror symmetry conjecture for variations of Hodge structures predicts  that
 the two differential equations are equivalent; in particular for a true mirror potential $W$ it should hold that $G_X=\pi_W$. 
 
The existence of mirror dual potentials has been established through explicit computation for important classes of Fanos like toric varieties and their complete intersections, and for certain other examples like Fano 3-folds where the proofs in a  sense rely on a case-by-case analysis, see e.g.~\cite{CoGi07,Prz13,ILP13,CCGK16}. Given a general Fano variety, little is known about the existence of a mirror dual potential, let alone mirror symmetry.
 The following is  a   conjecture.
  
 \begin{conjecture}
 	\label{conj:mirror_fano}
 	For every smooth compact Fano variety, there exists a Laurent polynomial which is its mirror dual.
 \end{conjecture} 

  Theorem~\ref{th:lg_per} reduces this  to the problem of finding a monotone Lagrangian torus in $X$, without ever having to compute the periods explicitly. One generally expects  that any Fano variety admitting a toric degeneration  contains such a torus.

  \begin{corollary}
  	Conjecture~\ref{conj:mirror_fano} holds for any Fano variety that contains a monotone Lagrangian torus.\qed
  \end{corollary}

 \subsection{Quantum Lefschetz}
One envisions another application of the periods theorem to enumerative geometry. Given a smooth divisor  $Y\subset X$ where both $X$ and $Y$ are Fano, the quantum Lefschetz theorem of  Coates and Givental~\cite{CoGi07} relates their quantum periods to each other; in \cite{CCGK16} one finds a more explicit version of this formula when $Y$ is toric.
Theorem~\ref{th:lg_per} should  lead to a symplectic-geometric understanding of this result.

Consider a Fano hypersurface $Y\subset X$ whose homology class is proportional to the anticanonical class. Assuming $Y$ contains a monotone Lagrangian torus $L\subset Y$, one can write down an explicit formula relating the potential of $L\subset Y$ with the potential of its Biran lift \cite{BiCi01,Bi06} to a monotone torus in $X$, cf.~\cite{Bk13} and the ongoing project \cite{DTVW17}. Comparing the periods of the two potentials yields a version of the quantum Lefschetz  theorem: the result has been checked to agree with 
\cite{CCGK16} in several illuminating examples. One hopes that the strong requirement on $Y$ above may be removed; the details will appear in a forthcoming work.

\subsection{Lagrangian tori and mirror cluster charts}
Mirror symmetry predicts that any compact Fano variety $X$ has a mirror Landau-Ginzburg model $(\check X,W)$ where $\check X$ is a complex variety, and $W\co \check X\to\C$ is a holomorphic function whose fibres are compact Calabi-Yau varieties. Here $\check X$ is a complex variety which, without the potential $W$, is expected to be mirror to $X\setminus \Sigma$ where $\Sigma\subset X$ is a smooth anticanonical divisor. Mirror symmetry is a package of several interrelated conjectures,  including mirror symmetry for variations of Hodge structures and homological mirror symmetry; they have been established in a significant number of cases.

The periods theorem and the results about mutations  suggest the following very geometric conjecture within the framework of mirror symmetry; cf.~\cite[Section~1.2]{To17}.

\begin{conjecture}
	\label{conj:charts}
	Let $X$ be a compact Fano variety.
	The set of embeddings $(\C^*)^n\subset \check X$ (`cluster charts') is in bijection with the set of monotone Lagrangian tori in $X$ up to Hamiltonian isotopy. For a chart $(\C^*)^n_L\subset \check X$ corresponding via this bijection to a monotone Lagrangian torus $L\subset X$, it holds that
	\begin{equation}
	\label{eq:W_chart}
	W|_{(\C^*)^n_L}=W_L\in\C[x_1^{\pm 1},\ldots,x_n^{\pm 1}].
	\end{equation}
	Above, the left hand side is the restriction of $W$ from $\check X$ onto $(\C^*)^n_L$; since this is a regular function on $(\C^*)^n_L$, it can be written as a Laurent polynomial. The right hand side is the LG potential of $L$.
\end{conjecture}

Rather than considering  cluster charts in the complex variety $\check X$, one can formulate an analogous conjecture by looking  at the chambers in the mirror constructed using the Gross-Siebert algorithm \cite{Gro11,GrSi11}; the papers of Carl, Pumperla and Siebert \cite{CPS11} and Prince \cite{Pri17} explain this point of view in detail.

This conjecture seems to be currently out of reach even for $X=\C P^2$. This is  because it is extremely hard to classify Lagrangian tori up to Hamiltonian isotopy even in the simplest manifolds. Basically, the only known non-trivial classifications are for $T^*T^2$ \cite{DGI16} and the neighbourhood of the Whitney immersion \cite{Riz17}. It is an open problem whether every monotone Lagrangian torus in $\C P^2$ is Hamiltonian isotopic to a Vianna torus; and whether a monotone Lagrangian torus in $\R^4$ is Hamiltonian isotopic to either the Clifford or the Chekanov one.

What Theorem~\ref{th:lg_per} makes more tractable is an algebraic version of Conjecture~\ref{conj:charts}:
that the set of LG~potentials of all possible monotone Lagrangian tori in $X$ is in bijection with the restrictions of the mirror LG model to the  charts,
$$
\{W_L:L\subset X\}/GL(n,\Z)=\{W|_{(\C^*)^n}:(\C^*)^n\subset\check X\}/GL(n,\Z).
$$
Given a sufficiently complete construction of Lagrangian tori in $X$,  Theorem~\ref{th:lg_per} can reduce this question to a combinatorial one.

\begin{conjecture}
	\label{conj:cp2}
Let $W(x,y)$ be  a Laurent polynomial in two variables. If it has the same periods as $x+y+1/xy$, that is $\pi_W=\pi_{x+y+1/xy}$, then $W$ is equal to the potential of a Vianna torus.\qed
\end{conjecture}

Under certain conditions on the Newton polytope of $W$, this conjecture has been solved by Akhtar and Kasprczyk \cite{AkKa13}.
Now Theorem~\ref{th:lg_per} implies the following.

\begin{corollary}
	Assume Conjecture~\ref{conj:cp2} folds. Then
every monotone Lagrangian torus in $\C P^2$ has the LG~potential of a Vianna torus.\qed
\end{corollary}
	
\subsection{Structure and proof ideas} The proof of Theorem~\ref{th:lg_per} begins with a beautiful idea of Cieliebak and Mohnke which applies SFT stretching around $L\subset X$ to the holomorphic spheres computing $\langle\psi_{d-2}\,\pt\rangle^\en_{X,d}$ choosing the point constraint to lie on $L$. The argument, sketched in Figure~\ref{fig:gw_stretch}, is recalled in Section~\ref{sec:cm}.  Sections~\ref{sec:bs},~\ref{sec:grav} analyse the two broken moduli spaces arising from the stretching; they compute the concurrently introduced Borman-Sheridan class and gravitational descendants, respectively. The short Section~5 revisits the stretching argument from Section~\ref{sec:cm} to show that it essentially amounts to
Theorem~\ref{th:gw_s_intro}. Section~\ref{sec:torus} contains the central computation of the paper: it determines  descendant invariants for the cotangent bundle of the $n$-torus. The computation of descendants and Theorem~\ref{th:gw_s_intro} imply Theorem~\ref{th:lg_per}.

The computation in Section~\ref{sec:torus} starts with an identity for descendants of $T^*T^n$ implied by the $L_\infty$ relations and the computed Chas-Sullivan bracket on the loop space of the torus. The section finishes with an inductive procedure which reduces all descendants to the ones coming from products of projective spaces.

\begin{figure}[h]
\includegraphics[]{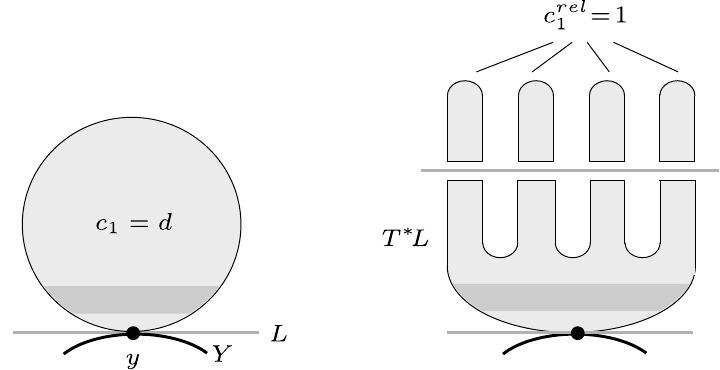}
\caption{The stretching argument. The actual argument uses a stabilising divisor and additional marked points, not shown here.}
\label{fig:gw_stretch}
\end{figure}
	
	\subsection*{Acknowledgements} I am very thankful to Chiu-Chu Melissa Liu and Kyler Siegel for pointing out a mistake in the previous version of the article.
	
	I am deeply grateful to Denis Auroux, Tobias Ekholm and Ivan Smith for creating environments where doing mathematics turns into a truly engaging process. I learned a lot from them, and benefited many times from their support.
	
	This work was partially supported by the Simons
	Foundation grant \#385573, Simons Collaboration on Homological Mirror
	Symmetry.
	
	\section{Stretching closed-string descendants}
	\label{sec:cm}
	\subsection{GW invariants with a tangency}
	\label{subsec:grav_inv}
	This section starts with a review of Gromov-Witten invariants with a tangency. This definition will soon be modified to a different one, leading to $\langle \psi_{d-2}\pt\rangle_{X,d}^\en$.  The lemmas in this subsection are provided for completeness.
	
	Let $X$ be a compact monotone symplectic manifold, e.g.~a smooth Fano variety
	with a K\"ahler symplectic form.
	Fix a point $y \in X$, a (compatible) almost complex structure $J$ on $X$
	which is integrable in a neighbourhood of $y$, and a germ $Y$ of a
	$J$-complex hypersurface defined in a neighbourhood of $y$ so that $y\in Y$.
	Following \cite{CM14} consider the moduli space
	\begin{equation}
		\label{eq:m_t_1}
		\M_1(\tau_{m-1}\,\pt)_d=\left\{
		\begin{array}{l}
			u\co \C P^1\to X,\ \bar \del u=0, \ c_1(u)=d,\\
			u(0)=y ,\\
			u \text{ has local intersection} \\
			\quad \text{multiplicity }m\text{ with }Y\text{ at }u(0)=y
		\end{array}
		\right\}/Aut(\C P^1,0).
	\end{equation}
	The intersection multiplicity condition appearing above is  sometimes called a \emph{tangency condition} of order $m-1$.
	The dimension of this moduli space is
	\begin{equation}
		\label{eq:dim_M_t_1}
		2d-2(m+1).
	\end{equation}
	Indeed, it can be computed as
	$$
	2n+2d-4-2n-2(m-1).
	$$
	Here
	$2n+2d-4$
	is the dimension of the space of Chern number~$d$ holomorphic spheres with one marked points; the
	$-4$ summand is due to the quotient by $Aut(\C P^1,0)$. The $(-2n)$ and the $-(m-1)$
	summands account for the incidence condition with $y$ and the intersection multiplicity, respectively. 
	
	The moduli space is $0$-dimensional whenever $m=d-1$, and for each $d\ge 2$ the \emph{Gromov-Witten invariant with a tangency} is defined by:
	\begin{equation}
		\label{eq:dfn_tau_1}
		\langle\tau_{d-2}\, \pt\rangle_{X,d}\coloneqq \#\M_1(\tau_{d-2}\,\pt)_d\in \Z.
	\end{equation}
	(The appearance of $X$ in the subscript will be frequently omitted.)
	To make sure that this invariant is well-defined, one has to argue that the
	moduli space (\ref{eq:dim_M_t_1}) is regular for generic $J$. As usual in
	Gromov-Witten theory, the regularity for generic $J$ is easily ensured for
	\emph{simple} curves; and it turns out that all curves in this moduli space are
	in fact simple.
	
	\begin{lemma}\label{lem:m_t_1_simple}
		For any $d\ge 2$, every curve in $\M_1(\tau_{d-2}\,\pt)_d$ is simple for a
		generic $J$.
	\end{lemma}
	
	\begin{proof}
		Suppose a curve $u\in \M_1(\tau_{d-2}\,\pt)_d$ is a degree $p$ branched cover
		over a simple curve $\tilde u$, and the cover $u\to\tilde u$ has ramification
		order $q\le p$ at the origin (the tangency point).
		Then $c_1(\tilde u)=d/p$ and $\tilde u$ intersects $Y$ with multiplicity
		$(d-1)/q$ at the origin. Hence $\tilde u$ belongs to the space
		$$
		\M_1\left(\tau_{\frac{d-1}{q}-1}\,\pt\right)_{d/p}.$$
		Since simple curves are regular for generic $J$, the dimension of this moduli space has to be $\ge 0$. In view of
		(\ref{eq:dim_M_t_1}) this means
		$$
		d/p-(d-1)/q-1\ge 0
		$$
		or equivalently
		$$
		(d-p)q\ge (d-1)p.
		$$
		But $q\le p$ so this implies $d-p\ge d-1$, or equivalently $p=1$. Hence $u$ is
		simple.
	\end{proof}

Below, the stable map compactification $\bM_1(\tau_{m-1}\,\pt)_d$
shall mean the following.
Consider the moduli space $\M_1(\pt)_d$
 of all Chern number~$d$ spheres with one marked point passing through $y\in X$, without any tangency condition. This space has the usual Kontsevich stable map compactification
  $\bM_1(\pt)_d$.
 There is an inclusion
 $$
 \M_1(\tau_{m-1}\,\pt)_d\subset \M_1(\pt)_d,
 $$ 
 and by
 $\bM_1(\tau_{m-1}\,\pt)_d$
 one denotes its closure 
 in  $\bM_1(\pt)_d$.

\begin{lemma}
	\label{lem:M_codim_2}
All compactification strata of  $\bM_1(\tau_{m-1}\,\pt)_d$ are  of complex codimension~1 and higher.
\end{lemma}

\begin{proof}
	The curves in $\M_1(\tau_{m-1}\,\pt)_d$ can undergo bubbling which may or may not involve the marked point of tangency.
	In the first case, it follows from monotonicity in a standard way that the bubbled configurations are of codimension~1 and higher.
		
	The second type of bubbling has one special case when the stable curve inheriting the tangency condition
	becomes contained in $Y$ near the tangency point $y$. If $Y$ were a globally defined divisor, this would have been an issue,
	but for a generic germ $Y$ of a hypersurface and generic $J$, there are no \emph{non-constant} rational curves in $X$ that pass through $y$ and are contained in $Y$ in the neighburhood of $y$. (This is ensured by perturbing $J$ in a non-integrable way in a bigger neighbourhood of $y$, keeping it integrable in a smaller one.)
	
	So it remains to consider the case when the component of the stable curve inheriting the tangency condition
	is a constant curve at $y$.  
Consider such a stable map $u\co C\to X$ modelled on a tree $T$ and denote by: 
\begin{itemize}
	\item 
	$e\in T$  the edge corresponding to the component containing $z_0$,
	\item $T'\subset T$ the maximal subtree containing $e$ and consisting of components that are all contracted to the point $y$, and $C_0$ the corresponding sub-curve,
\item   $C_1,\ldots,C_l$ the sub-curves of the domain curve corresponding to the connected components of $T\setminus T'$, 
	\item
	and $z_i\in C_i$ the marked points attached to the curves in $T'$ so that $u(z_i)=y$. 
\end{itemize}
By \cite[Lemma~7.2]{CM07}, the intersection multiplicity $m$ with $Y$ gets distributed between the curves $C_i$ namely:
$$
\sum_{i=1}^l u\cdot_{z_i} Y\ge m
$$
where $u\cdot_{z_i} Y$ is the intersection multiplicity at $z_i$.
Suppose that $u_\epsilon \in \M_1(\tau_{m-1}\,\pt)_d$ is a sequence of elements Gromov converging to a stable map $u\in \bM_1(\pt)_d$ as $\epsilon\to 0$.
Using the notation above, let $m_i=u\cdot_{z_i}Y$ and $d_i=c_1(C_i)$.  One has $\sum d_i=d$.
The dimension of the moduli space containing $C_i$ is at most
$$
2d_i-2m_i-2,
$$
where the equality is achieved when $C_i$ has a single component (no extra bubbles).
The sum of these dimensions is
$$
2d-2m-2l.
$$
The dimension of the moduli space of the ghost component $C$ mapping to the point $y$ is 
$$2(l+1)-6=2l-4;$$ this is the dimension of the moduli space $\M_{l+1}$ of $l+1$ marked points on $\C P^1$. The resulting total sum of dimensions is at most
$$
2d-2m-4=2d-2(m+1)-2,
$$
as claimed.
\end{proof}

\begin{lemma}
	\label{lem:gw_inv}
The number $\langle\tau_{d-2}\,\pt\rangle_{d}$ is invariant of the choice of $y$, $Y$ and $J$ as above.
\end{lemma}	
\begin{proof}
	Consider a 1-dimensional moduli space $\M_1(\tau_{d-2}\,\pt)_d$ corresponding to a generic homotopy of the data $y,Y,J$ in a 1-parametric family. By Lemma~\ref{lem:M_codim_2}, it undergoes no bubbling.   
\end{proof}	

	\subsection{Two stabilisations}
\label{subsec:stab}
This subsection is again mostly provided for completeness. It brings the setting closer towards the actual definition needed.

The domains of the curves from the moduli space (\ref{eq:m_t_1}) are unstable:
they only have one marked point. Although it was shown above that this is not a
problem for the regularity of the moduli space, for an
argument in Section~\ref{sec:pr} it will be convenient to work with curves over stable domains, i.e.~domains having at least 3 marked points.
Two natural ways of stabilising the moduli problem (\ref{eq:m_t_1}) will now be discussed, to spell out a numerical difference.

Fix an integer $p\ge 2$.
Let $\Sigma_i\subset X$, $i=1,\ldots,p$, be auxiliary oriented smooth
codimension~2 submanifolds Poincar\'e dual to $Nc_1(X)$. One does not require the
$\Sigma_i$ to be complex or even symplectic. Consider the moduli space
\begin{equation}
\label{eq:m_t_p_diff}
\M_{p+1}(\tau_{m-1}\,\pt,\Sigma_1,\ldots,\Sigma_p)_d=\left\{
\begin{array}{l}
u\co \C P^1\to X,\ \bar \del u=0, \ c_1(u)=d,\\
u(0)=y ,\\
u(z_i)\in \Sigma_i,\quad i=1,\ldots,p,\\
u \text{ has local intersection} \\
\quad \text{multiplicity }m\text{ with }Y\text{ at }u(0)=y
\end{array}
\right\}.
\end{equation}
Here $z_i\in \C P^1$ are distinct marked points which are free to vary in the
domain.
Recall that $Y$ is a germ of a hypersurface as earlier. For any $p$ this moduli
space has the same dimension as 
$\M_{p+1}(\tau_{m-1}\,\pt)_d$ given by (\ref{eq:dim_M_t_1}).
For any $d\ge 2$, $p\ge 2$ and $N\ge 1$ one has:
\begin{equation}
\label{eq:tau_comparison_many}
\langle\tau_{d-2}\,\pt\rangle_d=\frac{1}{(Nd)^p}\cdot
\#\M_{p+1}(\tau_{d-2}\,\pt,\Sigma_1,\ldots,\Sigma_p)_d.
\end{equation}
Indeed, consider the forgetful map
$$
\M_{p+1}(\tau_{d-2}\,\pt,\Sigma_1,\ldots,\Sigma_p)_d\to
\M_{1}(\tau_{d-2}\,\pt)_d.
$$
It has degree $(Nd)^p$ because the intersection number between every curve 
$u\in \M_{1}(\tau_{d-2}\,\pt)_d$
and $\Sigma_i$ equals $Nd$. Moreover the
preimages of those intersection points in $\C P^1$ are different for all
$i$ for generic $J$, since $\Sigma_i\cap \Sigma_j\subset X$ is a codimension~4
submanifold. Hence marking the preimage of any such intersection by $z_i$
recovers the full preimage of the forgetful map, and its degree is evidently
$(Nd)^p$.

A similar
argument is used to prove that the right hand side of
(\ref{eq:tau_comparison_many}), i.e.~the count of the moduli space
(\ref{eq:m_t_p_diff}), is invariant under generic choices of $y$, $Y$, $J$ and $\Sigma_i$.
For this one needs to rule out bubbling happening to
(\ref{eq:m_t_p_diff}) when these data are changed in a 1-parametric
family. Using monotonicity one easily rules out any non-constant sphere bubbles,
and the remaining case to consider is when two different marked points $z_i$,
$z_j$ collide producing a constant bubble.
Forgetting the constant bubble, the main part of the curve (the one inheriting the
tangency condition)  acquires an incidence condition to $\Sigma_i\cap
\Sigma_j$ which is a codimension~4 submanifold; this makes the Fredholm index of the
moduli problem drop by $2$, so this does not generically happen.

For the second stabilisation scheme, suppose $X$ is a Fano manifold with a K\"ahler symplectic form.
Fix an integer $p\ge 2$ and a \emph{single}
oriented smooth \emph{complex} codimension~1 divisor $\Sigma\subset X$
Poincar\'e dual to $Nc_1(X)$. 
Equip $X$ with an almost complex structure $J$ which preserves $\Sigma$ and is
integrable in a neighbourhood of $\Sigma$.
Consider the moduli space
\begin{equation}
\label{eq:m_t_p}
\M_{p+1}(\tau_{m-1}\,\pt,\Sigma,\ldots,\Sigma)_d=\left\{
\begin{array}{l}
u\co \C P^1\to X,\ \bar \del u=0, \ c_1(u)=d,\\
u(0)=y ,\\
u(z_i)\in \Sigma,\quad i=1,\ldots,p,\\
u \text{ has local intersection} \\
\quad \text{multiplicity }m\text{ with }Y\text{ at }u(0)=y
\end{array}
\right\}.
\end{equation}
As earlier, $z_i\in \C P^1$ are distinct marked points which are free to
vary in the domain.
For any $p$ this moduli space again has the same dimension as 
$\M_{p+1}(\tau_{m-1}\,\pt)_d$ given by (\ref{eq:dim_M_t_1}).
For any $d\ge 2$, $N\ge 1$ and $Nd\ge p\ge 2$ one has:
\begin{equation}
\label{eq:tau_comparison}
\langle\tau_{d-2}\,\pt\rangle_d=\tfrac{1}{(Nd)(Nd-1)\ldots(Nd-p+1)}\cdot
\#\M_{p+1}(\tau_{d-2}\,\pt,\Sigma,\ldots,\Sigma)_d.
\end{equation}
This is to be compared with (\ref{eq:tau_comparison_many}) which has a different
numerical factor.
Indeed, consider the forgetful map as earlier
$$
\M_{p+1}(\tau_{d-2}\,\pt,\Sigma,\ldots,\Sigma)_d\to \M_{1}(\tau_{d-2}\,\pt)_d.
$$
This time, 
it has degree 
$(Nd)(Nd-1)\ldots(Nd-p+1)$.
This is true because any curve  $u\in \M_{1}(\tau_{d-2}\,\pt)_d$
intersects $\Sigma$ at precisely $Nd$ points by positivity of intersections, and
any ordered subcollection of $p$ points among them (without repetitions) can be
labelled by the $z_i$.

As earlier, it is to be explained why the  collision of marked points $z_i=z_j$
does not generically occur. This is again a codimension~2 phenomenon but for a
reason slightly different than in the previous setting. This time, assuming
$z_i$ and $z_j$ collide creating a constant bubble, the main part of the curve
acquires an incidence condition with $\Sigma$ with a \emph{tangency}, i.e.~local
intersection multiplicity with $\Sigma$ at least~two \cite[Lemma~7.2]{CM07}.
The tangency condition drops the Fredholm index of the corresponding moduli
problem by~2, much like in the previous case.

\subsection{Enumerative descendants}
\label{subsec:psi_inv}
As before, let $X$ be a compact monotone symplectic manifold. This time, it is important to work with stable curves from the beginning. Fix one smooth divisor $\Sigma$ dual to $Nc_1(X)$, as in the second setting from the previous subsection. Fix a point $y\in X$,  $y\notin \Sigma$, and an almost complex structure $J$ preserving $\Sigma$ which is integrable in a neighbourhood of $y$. Fix a germ $Y$ of a
$J$-complex hypersurface in a neighbourhood of $y$, so that $y\in Y$. Consider the following moduli space for each $m\ge 1$.

\begin{equation}
\label{eq:m_psi}
\M_{Nd+1}(\psi_{m-1}\,\pt,\Sigma)_d=\left\{
\begin{array}{l}
0,z_1,\ldots,z_{dN}\in \C P^1, \\
u\co \C P^1\to X,\ (du-X_H\otimes \beta)^{0,1}=0, \ c_1(u)=d,\\
u(0)=y ,\\
u(z_i)\in \Sigma,\quad i=1,\ldots,dN,\\
u \text{ has local intersection} \\
\quad \text{multiplicity }m\text{ with }Y\text{ at }u(0)=y
\end{array}
\right\}.
\end{equation}

There are two differences from (\ref{eq:m_t_p}). First, the number of marked points is $dN$, which is the intersection number $[u]\cdot \Sigma$. That is, the preimages of all intersection points of $u(\C P^1)\cap \Sigma$ are marked by the $z_i$; this can be considered as a special case of (\ref{eq:m_t_p}). 

The second difference is the crucial one; it will be responsible for the fact that the counts become different from those of (\ref{eq:m_t_p}). Namely, the holomorphicity equation  is perturbed by a Hamiltonian term: $(du-X_H\otimes \beta)^{0,1}=0$. Here:

\begin{itemize}
	\item $H\co X\to \R$ is a function supported in a neighbourhood of $y$, and such that the Hamiltonian vector field $X_H$ is transverse to $Y$;
	\item $\beta$ is a 1-form on $\C P^1$ which depends on the modulus of the curve, i.e.~on the positions of the $z_i$, with the following property. It is supported in an annulus surrounding $0\in \C P^1$ such that the other marked points $z_i$ do not belong to it, and are located to the other side of the annulus with respect to $0\in \C P^1$. For concreteness,  one may assume that $\beta$ equals $dt$ times a cutoff function, where $t$ is a radial co-ordinate on the annulus: this gives a more common version of Floer's equation on the annulus.   
\end{itemize}

The provision that the $z_i$ are to the other side of $0\in\C P^1$ also prescribes where the annulus region should go under bubbling. Namely, the annulus always stays on the component containing the marked point $0$ under any domain degeneration, with one exception: when the bubbling happens at the point $0\in \C P^1$, and does not come from a domain degeneration at all.
Figure~\ref{fig:two_bubbles} shows these two possibilities.

\begin{figure}[h]
	\includegraphics{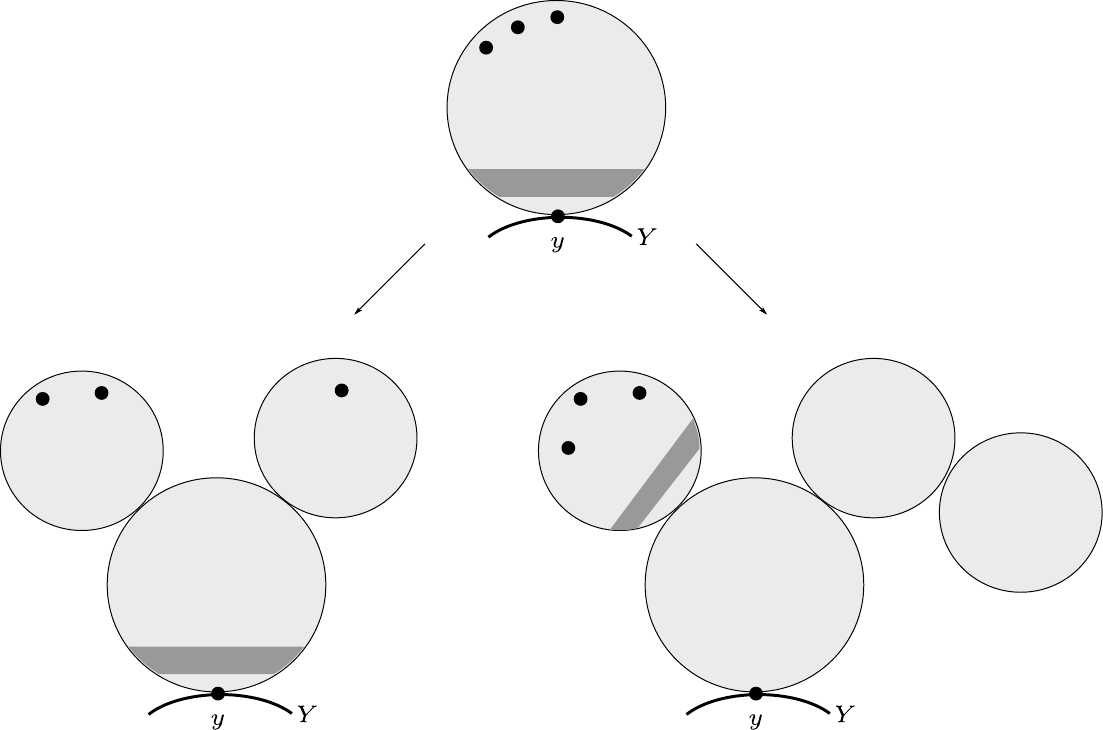}
	\caption{Two possibilities for bubbling at the tangency point $y$. The left one comes from domain degeneration, and the right one from bubbling at $y$. The shaded region contains the Hamiltonian perturbation. The marked points are the $z_i$, they are required to map to $\Sigma$.}
	\label{fig:two_bubbles}
\end{figure}

The lemma below says that curves in these moduli spaces cannot develop constant bubbles at the tangency point.

\begin{lemma}
	\label{lem:no_const_bub}
Whenever a curve in	
$\M_{Nd+1}(\psi_{m-1}\,\pt,\Sigma)_d$ bubbles, the component inheriting the marked point $0$ cannot be contained in $Y$.  
\end{lemma}

\begin{proof}
Consider the case shown	in Figure~\ref{fig:two_bubbles} (bottom left). In this case the component inheriting the marked point $0$ contains the annulus with a Hamiltonian perturbation. Because $X_H$ is transverse to $Y$, any map into $Y$ is not a solution.

Consider the case shown in  Figure~\ref{fig:two_bubbles} (bottom right). In this case the configuration contains a non-constant purely holomorphic curve without any of the marked points $z_i$. This curve intersects $\Sigma$ positively, by monotonicity. However, the marked points $z_i$ already realise the intersection number of the whole configuration with $\Sigma$, which gives a contradiction.
\end{proof}	

Define 
\begin{equation}
\label{eq:def_psi_t}
\langle \psi_{d-2}\pt\rangle_{X,d}^\en\coloneqq \frac 1{(dN)!}\cdot\# \M_{Nd+1}(\psi_{d-2}\,\pt,\Sigma)_d.
\end{equation}
One can directly show that this number does not depend on the choice of $y,Y,\Sigma$, and $J$ satisfying the above conditions. To prove the independence on $\Sigma$, one may argue as follows. Fix two divisors $\Sigma_1,\Sigma_2$, of degrees $N_1,N_2$ and assume that they form a normal crossings configuration, so that there is a $J$ preserving both of them. Consider the moduli space similar to (\ref{eq:m_psi}) with $dN_1+dN_2$ marked points, where $dN_1$ points are constrained to $\Sigma_1$ and $dN_2$ points are constrained to $\Sigma_2$. A rigid moduli space of this form has forgetful maps to the moduli spaces as in (\ref{eq:m_psi}) with respect to $\Sigma_1$ and $
\Sigma_2$ of degrees $(dN_1)!$ and $(dN_2)!$ respectively. This argument is to be compared with the previous subsection. The independence of (\ref{eq:def_psi_t})  on $\Sigma$ follows.

However, this argument is unnecessary, as the next subsection will prove that (\ref{eq:def_psi_t}) equals the actual $\psi$-invariant, up to a normalising factor.

\subsection{Comparison lemma}
Let $\bM_1(X,d)$ be the Kontsevich stable map compactification of the space of rational $J$-holomorphic curves $(u\co \C P^1\to X,\, z\in \C P^1)$ of Chern number~$d$ and with one marked point.
Let $\L\to \bM_1(X,d)$ be the line bundle whose geometric fibre at a stable map $(u\co C\to X,\, z)$ is $T^*_zC$.
The  gravitational descendant Gromov-Witten invariant of the point class is defined as follows:
\begin{equation}
\label{eq:psi_class_2}
\langle\psi_{d-2}\,\pt\rangle_{d}\coloneqq \int_{[\bM_1(X,d)]}c_1(\L)^{d-2}\cup ev^*([pt]),
\end{equation}
see \cite[Definition~10.11]{CoKa99}, \cite[Section~4]{Man99}. Recall that since $X$ is  monotone, $\bM_1(X,d)$ compactifies $\M_1(X,d)$ by  strata of positive complex codimension, and the  fundamental class $[\bM_1(X,d)]$ is defined in a straightforward way \cite{MDSaBook}. Above, the point class $[pt]\in H^{2n}(X)$ corresponds to the point constraint $y$ as in (\ref{eq:m_t_1}).

Let $\Sigma$ be a divisor dual to $Nc_1(X)$, as above.
By a version of the divisor axiom, compare with (\ref{eq:tau_comparison}):
$$
\langle\psi_{d-2}\,\pt,\Sigma,\ldots,\Sigma\rangle_{d}=(dN)!\cdot  \langle\psi_{d-2}\,\pt\rangle_{d}\cdot 
$$
The left hand side is defined using moduli spaces with $1+dN$ marked points free to vary in the domain, where $dN$ points are required to map to $\Sigma$, and the rest of the definition is analogous to (\ref{eq:psi_class_2}). Note that the left hand side is not quite a standard GW invariant: the standard one  would use $dN$ different copies of $\Sigma$ rather than one copy, in which case the factor in the above formula would be $(dN)^{dN}$ instead of $(dN)!$.

As a final modification, one introduces a Hamiltonian perturbation to the above moduli spaces to make them exactly as in (\ref{eq:m_psi}). This changes the moduli spaces by a cobordism, and $\langle\psi_{d-2}\,\pt,\Sigma,\ldots,\Sigma\rangle_{d}$ does not change, by the cobordism invariance of $c_1(\cL)$.
One arrives at the following.

\begin{equation}
\label{eq:psi_alt_3}
\langle\psi_{d-2}\,\pt\rangle_{d}=
\frac {1}{(dN)!}
\int_{[\M_{Nd+1}(\psi_{0}\,\pt,\Sigma)_d]}c_1(\L)^{d-2}.
\end{equation}

The integration happens over the moduli space (\ref{eq:m_psi}) with $m=1$. The incidence condition with $y$ (with no extra tangency condition) is already incorporated into $\M_{Nd+1}(\psi_{0}\,\pt,\Sigma)_d$, hence (\ref{eq:psi_alt_3}) has no $ev^*([pt])$ term. 
One is now in a position to compare (\ref{eq:def_psi_t}) with
(\ref{eq:psi_alt_3}).

\begin{proof}[Proof of Lemma~\ref{lem:desc}]
The proof elaborates on the discussion found in the paper of Gathmann~\cite{Gat02}, cf.~Vakil~\cite{Va99}.
For brevity, write for each $m\ge 1$:
$$
\M(m)\coloneqq \M_{Nd+1}(\psi_{m-1}\,\pt,\Sigma)_d
$$
for the moduli spaces from (\ref{eq:m_psi}), with a tangency condition at $y\in Y$ of order $m-1$, and a Hamiltonian perturbation as described above. Let $\bM(m)$ be its stable map compactification inside $\M(0)$.

It is shown in \cite[Section~2]{Gat02}
that for each $m\ge 1$, there is a section of
$\sigma_m$ of $\L^{\otimes m}$ over $\bM(m)$
which, by definition, computes the $m$th normal jet to $Y$ at $y\in Y$  of a holomorphic curve passing through $y$.
Recall that the Hamiltonian perturbation vanishes in the relevant region of the domain, so these jets are defined as in the purely holomorphic context.

It is easy to see that the zero-locus of $\sigma_m$ inside $\bM(m)$ equals $\bM(m+1)$.  
This fact crucially uses Lemma~\ref{lem:no_const_bub}: if there were boundary components of  $\bM(m)$ with a constant bubble inheriting the tangency point, $\sigma_m$ would vanish on such strata identically as well. Those strata need not belong to $\bM(m+1)$.

Using Lemma~\ref{lem:no_const_bub}, one can also show
that the zero loci of the sections $\sigma_1,\sigma_2,\ldots$ are generically transverse to each other.
Taking the intersection $\sigma_1^{-1}(0)\cap\ldots \cap \sigma_m^{-1}(0)$, one concludes:
$$
(m-2)!\cdot c_1(\L)^{m-2}\cdot [\bM_1]=[\M(m-1)].
$$
The factor $(m-2)!$ above is due to the fact that $\sigma_i$ is a section of $\L^{\otimes i}$ rather than $\L$. 
In paticular,
$$
(d-2)!\cdot c_1(\L)^{d-2}\cdot [\bM_1(\pt)]=[\M(d-1)].
$$
This amounts to Lemma~\ref{lem:desc}.
\end{proof}

\begin{remark}
	\label{rmk:comp_1}
	This remark essentially repeats what is mentioned in Siegel's \cite{Sie19}.
Although the relevant point has been stressed already, it can be helpful to discuss where the proof of Lemma~\ref{lem:desc} breaks down if one considers the moduli spaces of purely holomorphic curves (\ref{eq:m_t_p}) or (\ref{eq:m_t_1}), instead of (\ref{eq:m_psi}) which have a Hamiltonian perturbation. The key point is that Lemma~\ref{lem:no_const_bub} fails.
To see the consequence of it,
write
	$$
	\M(m)\coloneqq \M_{1}(\tau_{m}\,\pt,\Sigma)_d
	$$
	for the moduli spaces (\ref{eq:m_t_1}).
	
	The boundary strata of the stable map compactification  $\bM(m-1)$ inside $\bM(0)$ are again of complex codimension~1 and higher, see Lemma~\ref{lem:M_codim_2}.
	Now they come in two types: those strata where the component of the stable curve inheriting the tangency marked point is not a constant map to the point $y$, and those where it is. (The stable maps belonging to the strata of the second type were decribed in the proof of Lemma~\ref{lem:M_codim_2}.) 
	Let $$
	D(m-1)\subset \bM_1(m-1) 
	$$
	be the union of strata of the second type.
	The section $\sigma_m$ from the above proof now vanishes on
	$$
	\sigma_m^{-1}(0)=
	D(m-1)\cup \bM(m)\subset \bM(m-1).
	$$
	The common intersection of the zero-loci of the $\sigma_i$ computes the GW invariant with a tangency plus an extra contribution coming from the $D$-strata, which is not straightforward to determine. Hence the discrepancy between $\langle \tau_{d-2}
	\,\pt\rangle $ and $\langle \psi_{d-2}
	\,\pt\rangle$.  
\end{remark}

\begin{remark}
	\label{rmk:comp_2}
Consider what happens when the Hamiltonian perturbation in (\ref{eq:m_psi}) is turned off to zero. Holomorphic curves (\ref{eq:m_psi}) converge to purely holomorphic curves (\ref{eq:m_t_p}) computing the GW invariant with a tangency, or bubbled configurations of such curves. Generically, these extra bubbled configurations are unavoidable. This is because a version of Lemma~\ref{lem:M_codim_2} fails in the presense of the disappearing Hamiltonian term. Specifically, what fails is a version of \cite[Lemma~7.2]{CM07}: in the case of a constant bubble at $y$, the sum of intersection multiplicities of the non-constant purely holomorphic components with $Y$ may be smaller than the initial intersection multiplicity (\ref{eq:m_psi}). The reason is that nearby solutions of (\ref{eq:m_psi}) may have negative intersections with $Y$ near $y$ over the region supporting the Hamiltonian perturbation.
\end{remark}

	\subsection{Stretching after Cieliebak and Mohnke}
	\label{subsec:CM}
A beautiful stretching argument due to Cieliebak and Mohnke appearing in \cite{CM14} is now recalled. 
Let $X$	be a closed monotone symplectic $2n$-manifold and $L\subset X$ a monotone Lagrangian torus or, more generally, a monotone Lagrangian submanifold admitting a metric of non-positive sectional curvature.

Choose one of the two zero-dimensional moduli spaces, (\ref{eq:m_t_1}) or (\ref{eq:m_psi}), with $m=d-1$, and call it $\M(J)$ to remember the dependence on $J$. 
Recall the following result due to Charest and Woodward \cite[Theorem~3.6]{CW17}, building on previous work of Auroux, Gayet, Mohsen \cite{AGM01}, and Donaldson \cite{Do96}.
For every monotone Lagrangian submanifold $L$, there exists a smooth manifold $\Sigma\subset X$ which is disjoint from $L$ and dual to $Nc_1(X)\in H^2(X)$. Moreover, one can ensure that $\Sigma$ is dual to $2N\mu_L\in H^2(X,L)$ where $\mu_L$ is the Maslov class. If (\ref{eq:m_psi}) is used, a choice of such $\Sigma$ is assumed.

Choose the fixed point $y$ to lie in $L$: $y\in L$. As above, choose a $J$ which is integrable in a neighbourhood of $y$, and  a germ $Y$ of a hypersurface near $y$. For each $r>0$, consider the almost complex structure $J_r$ which is obtained by stretching $J$ along the boundary of a fixed Weinstein neighbourhood of $L$, using $r$ as the stretching parameter. See e.g.~\cite{EGH00,BEE,CM05} for the general stretching procedure, \cite{CM14} for the specific argument being explained, and e.g.~\cite{Hi04, HiLi14, DRE14,DGI16} for neck-stretching applied in related problems. Since the $J_r$ are unmodified near $y$, one may consider the rigid moduli spaces $\M(J_r)$ 
for each $J_r$. 

The SFT compactness theorem \cite{EGH00,BEE,CM05} states that as $r\to +\infty$,
the curves in the moduli space $\M(J_r)$
	converge to broken holomorphic buildings composed of punctured holomorphic curves in $T^*L$, $S^*L\times \R$, and $X\setminus L$.
Here $S^*L$ is the unit cotangent bundle of $L$.

There is a component $u$ of the limiting holomorphic building which lies in $T^*L$ and inherits from (\ref{eq:m_t_1}) resp.~(\ref{eq:m_psi}) the incidence condition at the fixed point $y\in L$ having intersection multiplicity $d-1$ with $Y$. If (\ref{eq:m_psi}) was chosen, this curve also inherits the annulus with a Hamiltonian perturbation. Let $p$ be the number of punctures of $u$ and $\gamma_1,\ldots,\gamma_p$ its (unparametrised) asymptotic Reeb orbits which correspond to closed geodesics on $L$ with respect to a  metric chosen in advance.  

Following \cite{CM14}, a simple action argument shows that the lengths of the geodesics that can potentially arise as asymptotic Reeb orbits from the above stretching are bounded by an apriori constant which depends on the size of the Weinstein neighbourhood of $L$ embeddable into $X$.
By the non-negative sectional curvature assumption, there exists a metric on $L$ all of whose closed geodesics up to any given length satisfy \cite{CM14}
$$
\mu(\gamma_i)\le n-1.
$$
Here $\mu$ stands for the Conley-Zehnder index, compare with Section~\ref{subsec:sh} below.
It also holds that $0\le \mu(\gamma_i)$, but this will not be used. 

Accounting for the tangency condition, the Fredholm index of the moduli problem satisfied by $u$ equals \cite{CM14}
$$
(n-3)(2-p)+\textstyle\sum_i\mu(\gamma_i)-(2n-2)-2(d-2)=p(3-n)+\textstyle\sum_i\mu(\gamma_i)\ge 0;
$$
it must be positive for regular curves. The regularity will be treated later in Section~\ref{sec:pr}; now it is taken for granted. In the above formula, $-2(n-2)$ accounts for the condition of passing through the fixed point $y\in T^*L$, and $-2(d-2)$ accounts for the tangency condition. 

Next, observe that 
$$
p\le d.
$$
Indeed, $T^*L$ has no contractible Reeb orbits so every puncture of $u$ gives rise to at least one holomorphic cap in $X\setminus L$ which topologically corresponds to a disk in $(X,L)$ of Maslov index at least~2, by the monotonicity of $L$. Since the original curve before the breaking was a Chern number~$d$ sphere, the limiting building has no more than $d$ caps.

Combining the inequalities in the three display formulas above, the only possible solution is found to be:
\begin{equation}
\label{eq:deg_cm_stretch}
p=d,\quad \mu(\gamma_i)=n-1.
\end{equation}
Furthermore it follows that the whole broken building looks as shown in Figure~\ref{fig:gw_stretch}, right. It consists of precisely $d+1$ components where one component is the curve $$u\subset T^*L$$ inheriting the tangency condition, and the remaining $d$ curves 
$$
w_1,\ldots,w_d\subset X\setminus L
$$
are holomorphic planes in $X\setminus D$ topologically corresponding to Maslov index~2 disks in $X\setminus D$; note that the curves $w_i$ are automatically simple. The $\gamma_i$ is the asymptotic orbit of $w_i$.

If  problem (\ref{eq:m_psi}) is chosen, the extra marked points $z_i$ are distributed among these planes, and the effect of this is analysed later.
Additionally, the curve $u$ inherits the annulus with a Hamiltonian perturbation.

This is where the  relevant part of the argument of Cieliebak and Mohnke stops: it proves at this point  that $L$ bounds a Maslov index~2 disk $w_i$, establishing  the Audin conjecture which was one of the main goals of \cite{CM14}. 
The next two sections give a second and more careful look at the corresponding moduli spaces; this leads to the definition of the Borman-Sheridan class and gravitational descendants of Liouville domains. 

\section{Borman-Sheridan class}
\label{sec:bs}
\subsection{Introduction}
 Liouville domains are a certain type of symplectic manifolds whose boundary $\del M$ is convex, hence is a contact manifold \cite{EGH00,BEE}.
Each Liouville domain has an associated completion: the result of attaching an infinite collar to $\del M$. For brevity, this paper usually does not distinguish between a Liouville domain and its completion; hopefully this does not create any confusion.

\begin{example}
Two important examples of Liouville domains are the cotangent bundle $T^*L$ of a smooth manifold $L$, and the complement $X\setminus\Sigma$ to an ample (not necessarily smooth) divisor in a compact K\"ahler manifold $X$.
\end{example}

Let $X$ be a closed monotone symplectic manifold, for example a compact Fano variety.
Let  $M\subset X$
be a symplectically embedded Liouville domain. For example, one may take $M=X\setminus \Sigma$ as in the previous example, or $M$ could be a Weinstein neighbourhood of a Lagrangian submanifold $L\subset X$.
Suppose $M$ satisfies the condition called \emph{monotonicity} introduced below which generalises the notion of a monotone Lagrangian submanifold. Then one can define the Borman-Sheridan class 
$$
\BS\in SH^0(M)
$$
and its $S^1$-equivariant version
$$
\BS\in SH^0_{S^1,+}(M).
$$
Roughly speaking, the Borman-Sheridan class controls the deformation of holomorphic curve theory as one passes from $M$ to $X$.
Its name originates from the ongoing work \cite{BS16}, and related ideas can be traced back to the earlier works of Fukaya \cite{Fuk06} and Cieliebak and Latschev \cite{CL09}. The first definition of the Borman-Sheridan class that appeared in the literature is due to Seidel \cite{Sei16}, and was given in the Calabi-Yau context. 

In the present setup, the Borman-Sheridan class has been defined by the author \cite{To17} using the Hamiltonian framework of symplectic cohomology (under a condition slightly stronger than monotonicity, although very similar).  This section gives a brief definition of the Borman-Sheridan class in the SFT framework. The SFT version is slightly easier to define, but assumes that a suitable virtual perturbation scheme for holomorphic curves has been fixed. 

In the special case when $\del M$ has no contractible Reeb orbits, the transversality difficulties of the  SFT framework disappear and virtual perturbations are not required. This is completely enough for the aims of this paper: the proof of Theorem~\ref{th:lg_per} only uses the Borman-Sheridan class for $M=T^*T^n$ which has no contractible Reeb orbits.

\subsection{Symplectic cohomology}
\label{subsec:sh}
Let $M$ be a Liouville domain with $c_1(M)=0$.  
An important symplectic invariant of $M$ is the symplectic cohomology $SH^*(M)$ and its relative, the positive equivariant symplectic cohomology $SH^*_{S^1,+}(M)$.
There are two  definitions of symplectic cohomology, using the Hamiltonian and the Symplectic Field Theory frameworks; both are quickly reminded.

\subsubsection*{Hamiltonian definition}
The classical definition of $SH^*(M)$ takes the direct limit of the Floer cohomologies of a cofinal family of time-perturbed Hamiltonians with fast linear growth near $\del M$. The references include~\cite{FH94, CFH95,  Oa04, Sei08,CFO10,Bourgeois2009,Ri13}. This paper adopts the cohomological grading convention of \cite{Ri13,Sei08}: the Floer differential has degree $+1$, there is a natural map
$H^*(M)\to SH^*(M)$, and the Viterbo isomorphism \cite{Vi96,SW06,ASc06,ASc10,Ab15} reads
$$
SH^*(T^*L)\cong H_{n-*}(\L L)
$$
where $\L L$ is the free loop space of $L$.  For any Hamiltonian in the cofinal family of sufficiently high slope, its 1-periodic orbit are called the generators of the Floer complex $CF^*(M)$. 

The generators of $CF^*(M)$ have the following geometric description: every unparametrised periodic Reeb orbit $\gamma\subset \del M$ gives rise to two Hamiltonian orbits $\hat{\gamma}$,  $\check \gamma$. Additionally, there are constant orbits corresponding to the critical ponts of the Hamiltonian on $M$.
The degree in $CF^*(M)$  is denoted by $|\cdot|$, and one has  $$|\check\gamma|=|\hat{\gamma}|+1.$$

\subsubsection*{SFT definition}
The second definition of $SH^*(M)$, using the SFT framework, is given in \cite{BO16,BEE,EO17}, see also \cite{Ek12,EkLe17,EHK17}. In \cite{BO16} the symplectic cohomology is denoted by $\overline{NCH}_*^{lin}(M)$ and called   \emph{filled  non-equivariant linearised contact homology}; while \cite{BEE,EO17} keep the notation $SH^*(M)$ which is adopted here. In this version the complex $CF^*(M)$ is generated by the critical points of a Morse function on $M$,  and two formal generators $\hat{\gamma}$, $\check{\gamma}$ associated to each (good) periodic Reeb orbit $\gamma\subset\del M$ (no longer thought of as being  1-periodic orbits of any specific Hamiltonian). The differential counts purely holomorphic cylinders between the Reeb orbits augmented by holomorphic planes, with respect to a $J$ which is cylindrical at  infinity of a domain. 

For a general Liouville domain $M$, the SFT definition of $SH^*(M)$ requires a choice of a virtual perturbation scheme making the moduli spaces regular; see \cite{FO3Book,Ho04,Baho15,Pa16}. An example of a Liouville domain whose SFT version of symplectic cohomology can be defined using elementary transversality arguments is $T^*L$ where $L$  is a smooth manifold admitting a metric of non-positive curvature, for instance, the torus \cite{BO09}. 

The  domains of curves whose counts define various operations on the symplectic cohomology, like the Floer differential or the product, by definition carry asymptotic markers at each puncture. This is true both in the Hamiltonian and the SFT setups. In the Hamiltonian setup the markers are required to define the Floer equation with a time-dependent Hamiltonian. In the SFT setup, asymptotic markers are not needed to define the holomorphic equation \emph{per~se}; the way they appear in moduli problems is as follows. Whenever $\check\gamma$ is the input asymptotic of a holomorphic curve, or $\hat\gamma$ is an output, the asymptotic marker on the curve at that cylindrical end must go to a given point on the geometric Reeb orbit $\gamma$, fixed in advance. In the other two scenarios, when $\hat\gamma$ is an input or $\check{\gamma}$ is the output, the asymptotic marker is unconstrained.

The grading conventions for symplectic cohomology in the SFT setup \cite{BO16} are different from those used in the Hamiltonian one. It is worth spelling out the difference.
One denotes the Conley-Zehnder index of an unparametrised Reeb orbit $\gamma\subset\del M$ by $$\mathrm{CZ}(\gamma)=\mu(\gamma).$$ When $M=T^*L$ and $\gamma$ corresponds to a geodesic in $L$, $\mu(\gamma)$ is the Morse index of that geodesic. When $L$ has non-negative curvature, the Morse indices lie within the range $[0,n-1]$, $n=\dim L$.
The SFT grading of $\gamma$ is given by $\bar\mu(\gamma)$ where
$$
\bar\mu(\gamma)=\mu(\gamma)+n-3,
$$
if $\dim M=2n$.
For example, $\bar\mu(\gamma)$ is the dimension of the moduli space of unparametrised holomorphic planes in $M$ asymptotic to $\gamma$.
Finally, one has
\begin{equation}
\label{eq:degree_dic}
\begin{array}{l}
|\check{\gamma}|=n-\mu(\gamma)=2n-3-\bar\mu(\gamma),\\
|\hat{\gamma}|=n-1-\mu(\gamma)=2n-4-\bar\mu(\gamma).
\end{array}
\end{equation}
where $|\cdot|$ is the cohomological degree in the Hamiltonian setup, used in this paper.

\subsubsection*{Positive $S^1$-equivariant symplectic cohomology}
This version of symplectic cohomology can again be defined using either the Hamiltonian or the SFT framework. The latter version  is also called {\it linearised contact homology}; see e.g.~\cite{BO09,BO16}. This paper uses the SFT version, which is now recalled. 

The complex $$CF^*_{S^1,+}(M)=CC^*_{lin}(M)$$ 
has one generator $\gamma$ associated to each (good) periodic Reeb orbit $\gamma\subset\del M$.  The differential counts purely holomorphic cylinders between the Reeb orbits augmented by holomorphic planes, this time without asymptotic markers. If one restricts the differential on $CF^*(M)$ considered in the previous subsection to non-constant $\hat \gamma$-orbits, one recovers the differential for $CF^*_{S^1,+}(M)$. The grading used in this paper is: $$|\gamma|=|\hat\gamma|=n-1-\mu(\gamma),$$
so that a version of the Viterbo isomorphism reads: 
$$SH^*_{S^1,+}(T^*L)\cong H_{n-*-1}(\L L/S^1,L).$$
In fact, this isomorphism will not be used here; the relevant computations will go via the non-equivariant version of the Viterbo isomorphism.

As in the non-equivariant case,
the SFT definition of $CF^*_{S^1,+}(M)$ in general requires a virtual perturbation scheme, but is given using elementary transversality arguments for a narrower class of Liouville manifolds which includes $T^*T^n$. This is enough for the purpose of proving Theorem~\ref{th:lg_per}.

\subsection{Monotone Liouville subdomains}
Let $(X,\omega)$ be a monotone symplectic manifold, and $(M,\theta)\subset X$ a (non-completed) Liouville domain symplectically embedded into $X$, where $\theta$ is a Liouville form on $M$. Assume that $c_1(M)=0$ and choose a trivialisation of the canonical bundle $K_M$. This trivialisation gives rise to the relative first Chern  class
$$
c_1^{rel}\in H^2(X,M)
$$
analogous to (twice) the Maslov class of a  Lagrangian submanifold. 
Define 
$$\omega^{rel}\in H^2(X,M)$$
by its value on any relative homology class $A\in H_2(X,M)$ as follows:
\begin{equation}
\label{eq:omega_rel}
\omega^{rel}=
\int_A \omega-\int_{\del A}\theta.
\end{equation}

\begin{definition}
	\label{def:m_monot}
A symplectic  embedding $(M,\theta)\subset X$ of a Liouville domain with boundary into a closed symplectic manifold is called \emph{monotone} if 
$
c_1^{rel}
$
and $\omega^{rel}$ are positively proportional.
\end{definition}

\begin{remark}
	If $(M,\theta)\subset X$ is monotone and $\theta'$ is a different Liouville form on $M$ such that $\theta-\theta'$ is an exact 1-form, then $(M,\theta')\subset X$ is also monotone.
\end{remark}

\begin{example}
	\label{ex:l_m_mon}
Suppose $L\subset X$ is a monotone Lagrangian submanifold, and let $M=T^*L$ be (a neighbourhood of the zero-section of) its cotangent bundle embedded into $X$ as a Weinstein neighbourhood of $L$, equipped with the Liouville form $\theta$ making the zero section exact. Recall the agreement to blur the distinction between a Liouville domain and its completion.  Then $M\subset X$ is a monotone embedding. Indeed, a class $A\in H_2(X,M)$ can be uniquely extended to a class  $A'\in H_2(X,L)$ so that $\omega^{rel}(A)=\omega(A')$ and $c_1^{rel}(A)=2\mu(A')$.
\end{example}

\begin{example}
Suppose $\Sigma\subset X$ is a (not necessarily smooth) anticanonical divisor, then $M=X\setminus \Sigma$ is monotone in $X$ when equipped with the standard trivialisation of $K_M$ and the standard Liouville form $\theta$ having a simple pole along $\Sigma$. Indeed,  a version of the Stokes formula gives
$$
\int_A\omega=\int_{\del A}\theta +A\cdot \Sigma
$$
and $A\cdot\Sigma=c_1^{rel}(A)$.
\end{example}

\subsection{Borman-Sheridan class}
\label{subsec:bs}
Let $(M,\theta)$ be a monotone Liouville subdomain in a closed symplectic manifold $X$.
For a Reeb orbit $\gamma\subset\del M$, consider the corresponding generator $\gamma\in CF^*_{S^1,+}(M)$ and assume that it has degree~zero: $|\gamma|=0$.
Set
$$
W=X\setminus M,
$$
with $\del M$ being its negative contact boundary.
In this subsection is better to distinguish the notation between $W$ and its completion
$$
\hat W=
\left(  (-\infty,0]\times \del  M\right)\sqcup (X\setminus M) 
$$
with the symplectic form
\begin{equation}
\label{eq:tilde_omega}
\tilde \omega=
\begin{cases}
d(e^r\lambda)&\textit{on } (-\infty,0]\times \del  M
\\
\omega&\textit{on } X\setminus M,
\end{cases}
\end{equation}
where $\lambda=\theta|_{\del M}$ is the contact form on $\del M$, $r\in(-\infty,0]$ is the standard collar co-ordinate and $\omega$ is the initial symplectic form on $X$.

Assuming that $\del M$ has no contractible Reeb orbits, define
$$		
\label{eq:m_bs}
		\M_{1|0}( \gamma)_W=\left\{
		\begin{array}{l}
			u\co \C\to \hat W,\ c_1^{rel}(u)=1,\\
			u\text{ is asymptotic to } \gamma \text{ at }\infty\text{ as output} 
		\end{array}
		\right\}/\mathit{Aff}(\C).
$$

When $\del M$ has contractible Reeb orbits, the moduli space $\M_{1|0}(\gamma)_W$ is defined to consist of curves as above which are additionally augmented, i.e.~they may have arbitrarily many additional punctures augmented by holomorphic planes in $M$. The details follow the usual framework of augmented curves~\cite{BO16,BEE} and are omitted.

The moduli space $\M_{1|0}(\gamma)_W$ is zero-dimensional.
In general, the dimension formula for $\M_{1|0}(\gamma)_W$ letting $| \gamma|$ and $c_1^{rel}(u)$ be arbitrary, would be
$$
2c_1^{rel}(u)+| \gamma|-2.
$$
The \emph{Borman-Sheridan class} is defined as follows:
$$
\BS=\sum_{\gamma}(\#\M_{1|0}( \gamma)_W)\cdot \gamma\ \in\  SH^0_{S^1,+}(M).
$$
It is closed because its Floer differential counts the boundary points of a one-dimensional moduli space of the same type, which causes the differential to vanish.
This class depends on $M$ and $W$, i.e.~on $M$ and its embedding into $X$.
The monotonicity  condition implies the invariance of the Borman-Sheridan class; 
this is an analogue of the statement that the count of Maslov index~2 disks with boundary on a monotone Lagrangian submanifold is an invariant. A preliminary lemma is necessary first.

\begin{lemma}
	\label{lem:areas}
		Consider a holomorphic curve $u\co C\to \hat W$ asymptotic to a Reeb orbit and its relative homology class
		$$[u]\in H_2^{BM}(\hat W)\cong H_2(X,M)$$
		where $BM$ stands for Borel-Moore homology, i.e.~the homology of locally compact chains.
		Under this identification, it holds that
		$$
		\tilde \omega(u)=\omega^{rel}(u)
		$$
		where the two symplectic forms are from (\ref{eq:tilde_omega}) and (\ref{eq:omega_rel}).
\end{lemma}

\begin{proof}
Consider a collar $C\subset X\setminus M$ of $\del M$  and identify it with
$$
C= (0,\epsilon)\times \del M,\quad \omega|_C=d(e^r\lambda)
$$	
where $r\in (0,\epsilon)$ and $\lambda=\theta|_{\del M}$. The same collar embeds into $\hat W$ extending the infinite end of $\hat W$ to an embedding
$$
\tilde C= (-\infty ,\epsilon)\times \del M\subset\hat W,\quad \tilde \omega|_{\tilde C}=d(e^r\lambda).
$$
Fix  a monotone bijective function $ (-\infty,\epsilon)\to (0,\epsilon)$ which is the identity in a neighbourhood of $\epsilon$. It defines a diffeomorphism $\tilde C\to C$ which extends by the identity to the diffeomorphism $\phi\co \hat W\to X\setminus M$.
The claim is that for any 
relative homology class $A\in H_2(X,M)=H_2(X\setminus M,\del M)$, it holds that
\begin{equation}
\label{eq:phi_om}
\phi_*\tilde \omega(A)=\omega^{rel}(A).
\end{equation}
To see this, realise  $A$ as  a chain and break it into a union $A'\sqcup A''$ of chains in $C$ and $(X\setminus M)\setminus C$. One has  $\int_{A''}\omega=\int_{A''}\phi_*\tilde \omega$ because the two forms coincide on $(X\setminus M)\setminus C$. Next one has $\int_{A'}\omega=\int_{A'}\phi_*\tilde \omega+\int_{\del A}{\lambda}$ by the Stokes formula. But $\int_{\del A}{\lambda}=\int_{\del A}{\theta}$, which justifies (\ref{eq:phi_om}). This equality immediately implies the statement of the lemma.
\end{proof}

\begin{lemma}
	If $M\subset X$ is monotone,
the Borman-Sheridan class $\BS\in SH^0_{S^1,+}(M)$ does not depend on the choice of $J$ on $\hat W$ which is compatible with $\tilde \omega$ and cylindrical at infinity.
\end{lemma}

\begin{proof}
Any holomorphic curve $u\co\C\to\hat W$ satisfies $\tilde \omega(u)> 0$; by Lemma~\ref{lem:areas}, $ \omega^{rel}(u)> 0$.
So the monotonicity condition guarantees that  $c_1^{rel}(u)\ge 1$. This implies that 1-dimensional moduli spaces (\ref{eq:m_bs}) with varying $J$ do not undergo any SFT breaking except for the breaking of augmented Floer cylinders. Indeed, any other broken building would have at least two components mapping into $W$, and this would contradict the additivity of $c_1^{rel}$ and the fact that  the Borman-Sheridan class is defined using curves with $c_1^{rel}=1$.
\end{proof}	

\begin{figure}[h]
	\includegraphics[]{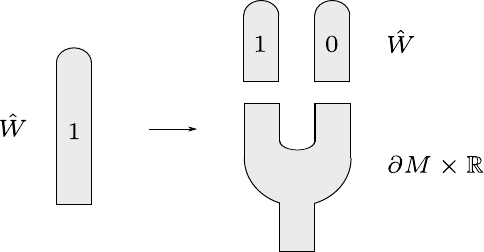}
	\caption{Bubbling responsible for the failure of the invariance of the Borman-Sheridan class in the non-monotone case. The numbers indicate $c_1^{rel}$.}
	\label{fig:bs_bubble}
\end{figure}

\begin{remark}
For a non-monotone Lagrangian submanifold $L\subset M$, 
a
Maslov index~2 disk can break into a Maslov~index~2 and a Maslov index~0 disk inside a 1-parametric family. A similar phenomenon for the Borman-Sheridan class is the breaking of a (relative) Chern number~1 plane into a Chern number~1 plane and a Chern number~0 plane. It may be useful to take a look at a model for this breaking,  shown in Figure~\ref{fig:bs_bubble}.
\end{remark}

\begin{remark}
The definition of the Borman-Sheridan class using Hamiltonians appearing in \cite{To17} uses a slightly stronger assumption on the embedding $M\subset X$ than monotonicity. The condition is that there is a smooth Donaldson divisor $\Sigma\subset X$ away from $M$, $M\subset X\setminus\Sigma$ is a Liouville embedding, and $c_1^{rel}\in H^2(X,M)$ is positively proportional to the intersection number with $\Sigma$. 
\end{remark}

\begin{lemma}
The Borman-Sheridan class lies in the image of $$SH^0_+(M)\to SH^0_{S^1,+}(M).$$
\end{lemma}

\begin{proof}
The non-equivariant version of the Borman-Sheridan class can be defined by counting the following moduli spaces:
	$$		
	\M_{1|0}( \hat \gamma)_W=\left\{
	\begin{array}{l}
	u\co \C\to \hat W,\ c_1^{rel}(u)=1,\\
	u\text{ is asymptotic to } \gamma \text{ at }\infty\text{ as output},\\
	u(\R_+) \text{ is asymptotic to a fixed point on } \gamma
	\end{array}
	\right\}/\mathit{Aff}_+(\C).
	$$
Theren is an obvious bijection between these moduli spaces and (\ref{eq:m_bs}),	
after relabelling the generators $\hat\gamma\mapsto \gamma$.	
\end{proof}	

\subsection{The torus case}
Consider a monotone Lagrangian torus $L\subset X$, and let
$M=T^*L$ be embedded into $X$ as a Weinstein neighbourhood of $L$. In this case the Borman-Sheridan class is nothing but a reformulation of the LG potential.
One has
\begin{equation}
\label{eq:SH_T_Laur}
SH^0(T^*T^n)\cong H_{n}(\L T^n)\cong\C[x_1^{\pm 1},\ldots,x_n^{\pm 1}].
\end{equation}
The first isomorphism is the Viterbo isomorphism, and the second one is understood by looking at the Serre fibration
$$
\Omega T^n\to \L T^n\to T^n
$$
where $\Omega T^n$ is the based loop space of $T^n$. Consider the map induced by intersection with the fibre of this fibration:
$$
i_!\co H_n(\L T^n)\to H_0(\Omega T^n).
$$
This is easily seen to be an isomorphism, and
$$
H_0(\Omega T^n)=\Z[\pi_1 T^n ]\cong \C[x_1^{\pm 1},\ldots, x_n^{\pm 1}].
$$
Next, one also has
\begin{equation}
\label{eq:SH_eq_T_Laur}
SH^0_{S^1,+}(T^*T^n)\cong \Z[\pi_1 T^n ]\cong\C[x_1^{\pm 1},\ldots,x_n^{\pm 1}]
\end{equation}
as vector spaces. This can be seen from the exact sequence 
$$\ldots \to SH^*_+(T^*T^n)\to SH^{*-1}_{S^1,+}(T^*T^n)\to SH^{*+1}_{S^1,+}(T^*T^n)\to\ldots$$
from \cite{BO16} (rewritten using the current grading convenstion) and the vanishing of the Euler class appearing in the map
$SH^{*-1}_{S^1,+}(T^*T^n)\to SH^{*+1}_{S^1,+}(T^*T^n)$.
Alternatively, one can choose a perturbaition of the standard Morse-Bott Reeb flow for which $CF^*_{S^1,+}(T^*T^n)$ is minimal up to an arbitrarily high energy truncation, and isomorphic to $\Z[\pi_1(T^n)]$ in degree zero.
\begin{lemma}
	\label{lem:bs_w}
Let $L\subset X$ be a monotone Lagrangian torus, $W_L$ its LG potential, and $\BS\in SH^0_{S^1,+}(T^*L)$ the Borman-Sheridan class for the Weinstein neighbourhood of $L\subset X$.
Under the identification (\ref{eq:SH_eq_T_Laur}), it holds that:
$$\BS=W_L\in\C[x_1^{\pm 1},\ldots, x_n^{\pm 1}].$$
\end{lemma}

\begin{proof}
Consider the non-equivariant Viterbo isomorphism. It is proved by considering moduli spaces of semi-infinite cylinders $[0,+\infty)\times S^1\to T^*L$ asympotic to generators $x\in CF^*(T^*L)$ and having free Lagrangian boundary condition on the 0-section $L$, giving rise to a quasi-isomorphism $CF^*(L)\to C_{n-*}(\L L)$. 

This holds for a general $L$, and now recall that $L$ is a torus. 
Fix a point $p\in L$ and a generator $\hat \gamma\in SH^0(T^*L)$ corresponding to a homology class $[\gamma]\in H_1(L;\Z)$. It follows from the Viterbo isomorphism that the count of the cylinders with input $\hat \gamma$ and boundary passing through $p$, equals $1$, moreover the boundary homology class of those cylinders is $[\gamma]$. Since the asymptotic marker for the $\hat \gamma$-orbit is unconstraining, this is equal to the count of analogous cylinders without an asymptotic marker.

Gluing these cylinders (see e.g.~\cite{Pa15} for the gluing in the regular setting) to the moduli space $\M_{1|0}(\gamma)_W$, one gets a sign-preserving bijection onto the moduli space of Maslov~index~2 holomorphic disks in $(X,L)$ passing through $p\in L$ and having boundary homology class $[\gamma]$.  The latter spaces count $W_L$ by definition.
\end{proof}	

\section{Gravitational descendants of Liouville domains}
\label{sec:grav}
This section
recalls the $L_\infty$ structure on the non-equivariant and the equivariant Floer complex; 
defines gravitational descendants of a Liouville domain, which are chain-level operations on its equivariant Floer complex; 
states the theorem computing gravitational descendants for the cotangent bundle of the $n$-torus; and mentions the string topology perspective of the story.

\subsection{Compactifying spaces of marked curves}
\label{subsec:dom_mark}
The first step is to introduce the moduli spaces of  domains used in the definitions of the $L_\infty$ structure and gravitational descendants. Recall the notation
$$
\begin{array}{ll}
\mathit{Aff}(\C)&=\{z\mapsto az+b\ |\ a\in \C,\ b\in\C\}
,\\
\mathit{Aff_+}(\C)&=\{z\mapsto az+b\ |\ a\in \R_+,\ b\in\C\}.
\end{array}
$$
For $k\ge 2$ consider the space $\M_{k+1}$ of $k+1$ distinct numbered marked points $$z_0,z_1,\ldots,z_k\in \C P^1$$ up to $Aut(\C P^1)$.
An \emph{marker} at a point $z_i$ is the choice of a real ray in $T_{z_i}\C P^1$; automorphisms of $\C P^1$ naturally act on  markers.

 By finding an automorphism of $\C P^1$ sending $z_0$ to the fixed point $\infty\in\C P^1$ one identifies
\begin{equation}
\label{eq:m_k_aff}
\M_{k+1}=\mathit{Conf}(\C,k)/\mathit{Aff}(\C),
\end{equation}
the configuration space of $k$ distinct points in $\C$ up to automorphisms of $\C$.
The following is an observation from \cite{EO17}. 

\begin{lemma}
	\label{lem:marker}
The choice of a marker at $z_0$ across the space $\M_{k+1}$ canonically induces a marker at each other  point $z_i$.
\end{lemma}

\begin{proof}
For each curve $u\in \M_{k+1}$ with a given marker at $z_0$, consider any automorphism of $\C P^1$ taking $z_0$ to $\infty\in \C P^1$ and the marker to $\R_+$. Once such an automorphism has been applied, assign the marker $\R_+$ to each other marked point $z_i$.
Automorphisms of $\C P^1$ preserving $\infty$ with the marker $\R_+$ form the group 
$\mathit{Aff_+}(\C)$. The action of this group preserves the horizontal direction, which ensures that the above assignment of markers is well-defined.
\end{proof}

The space $\M_{k+1}$ has the classical Deligne-Mumford compactification $\overline \M_{k+1}$ by stable curves. The boundary strata of this compactification are the unions of $\M_j$s for $j\le k$; they are of real codimension~2 and higher.

Now consider 
\begin{equation}
\label{eq:R_k}
\cR_{k+1}=\mathit{Conf}(\C,k)/\mathit{Aff_+}(\C)\cong S^1\times \M_{k+1}.
\end{equation}

\begin{remark}
It is useful to think of this space
as the moduli  space of $k+1$ distinct points in $\C P^1$ equipped with a marker at the first point $z_0$. 
\end{remark}

The space $\cR_{k+1}$ has a natural compactification $\overline{\cR}_{k+1}$.
Geometrically, as two or more marked points in $\cR_{k+1}$ approach each other, they create a bubble with a canonically induced marker at the attaching point which points in the direction $\R_+$. For example, the collision of two points together is a codimension~1 phenomenon because this collision can happen in an $S^1$-worth of ways with respect to the marker at infinity, see Figure~\ref{fig:l_comp}. More than two points may also collide in codimension~1.

\begin{figure}[h]
	\includegraphics[]{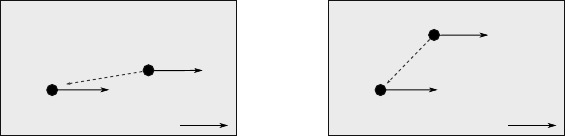}
	\caption{An $S^1$ worth of ways for two marked points to collide with respect to the asymptotic marker at infinity.}
	\label{fig:l_comp}
\end{figure}

Formally, the compactification strata of 
$\overline{\cR}_{k+1}$ are indexed by rooted trees $T$ with $k$ labelled leaves (not counting the root). A stratum corresponding to such a tree is the product of the spaces $\cR_{k_i+1}$ across the vertices $v_i\in T$, where $k_i+1$ is the valency of $v_i$:
$$
{\cR}_{k+1}^T=\prod_{v_i\in T}\cR_{k_i+1}.
$$

\subsection{$L_\infty$ structure on the Floer complex}
\label{subsec:l_infty}
Let $M$ be a Liouville domain with $c_1(M)=0$, and $CF^*(M)$ be the Floer complex computing $SH^*(M)$ graded cohomologically as in Section~\ref{subsec:sh}.
For each $k\ge 1$ there is an operation
$$
l_{neq}^k\co CF^*(M)^{\otimes k}\to CF^*(M),\quad \deg l_{neq}^k=3-2k,
$$
where $l_{neq}^1$ is the Floer differential. Here, {\it neq} stands for `non-equivariant'.
These operations turn $CF^*(M)$ into an \emph{$L_\infty$ algebra}, which means that they satisfy the following two properties. The first one is graded-commutativity: for a permutation $\sigma\in S_k$, one has
$$
l^k_{neq}(x_1,\ldots,x_k)=(-1)^\dagger\, l^k_{neq}(x_{\sigma_1},\ldots,x_{\sigma_k})
$$
where
$$
\dagger=\sum_{\stackrel{i<j:}{\sigma(i)>\sigma(j)}}|x_i|\cdot |x_j|.
$$
The second property is the $L_\infty$ relations:
\begin{equation}
\label{eq:l_inf}
\sum_{\stackrel{1\le r\le k,}{\sigma\in S_k}}(-1)^\maltese\, \tfrac{1}{r!(k-r)!}\, l_{neq}^{k+1-r}(l_{neq}^{r}(x_{\sigma_1},\ldots x_{\sigma_{r}}),x_{\sigma_{r+1}},\ldots, x_{\sigma_{k}})=0
\end{equation}
where an explicit formula for the sign number $\maltese$ is rather lengthy; the easiest way to specify the sign is to reformulate the $L_\infty$ relations in terms of a single equation on the bar complex in the way it is done in e.g.~\cite{Lat14,FO3Book}, cf.~\cite{Fuk01}. The references for this $L_\infty$ structure are \cite{Fuk06,Lat14,Fab13,EO17}. The latter reference defines the dual $L_\infty$ coalgebra, but the $L_{\infty}$ algebra definitions are analogous. See also~\cite{Sei14,PL16} for $l^2_{neq}$, the Lie bracket.

Consider the moduli space
$$
\M_{1|k}(x_0;x_1,\ldots,x_k)
$$ 
of holomorphic curves in $M$ with domain an element of $\cR_{k+1}$, where the output puncture is equipped with an asymptotic marker from the definition of $\cR_{k+1}$,
and each input puncture is equipped the induced asymptotic marker by Lemma~\ref{lem:marker}. The punctures  are asymptotic to $x_i\in CF^*(M)$ at the input punctures and to $x_0$ at the output, with the corresponding asymptotic marker constraints. The dimension formula reads
$$
\dim \M_{1|k}(x_0;x_1,\ldots,x_k)=\dim \cR_{k+1}+|x_0|-\sum_{i=1}^k|x_i|=2k-3+|x_0|-\sum_{i=1}^k|x_i|.
$$
The operations $l^k_{neq}$ are the counts of these moduli spaces when they are 0-dimensional. The graded-commutativity is immediate and the $L_\infty$ relations are derived by looking at the boundaries of the 1-dimensional moduli spaces above, compactified over $\overline\cR_{k+1}$. The orientation analysis responsible for the signs has been explained in \cite{EO17}, in the dual coalgebra setting.

\subsection{$L_\infty$ structure on the equivariant Floer complex}
There is a similar $L_\infty$ structure on the positive equivariant Floer complex:
$$
l^k\co CF_{S^1,+}^*(M)^{\otimes k}\to CF_{S^1,+}^*(M),\quad \deg l^k=3-2k.
$$
The definition, see e.g.~\cite{Sie19}, is analogous to the above with the difference that asymptotic markers are now not used, and the underlying spaces of domains are $\M_{k+1}$, not $\cR_{k+1}$. 

Let $\hat \gamma_i$ be a collection of non-constant hat-generators of the non-equivariant complex $CF^*(M)$. Since these are hat-type generators, whenever they serve as inputs for a  non-equivariant $L_\infty$ operation, all input asymptotic markers of the corresponding curves are unconstrained. So the output of this operation must have constrained asymptotic marker; otherwise the corresponding moduli space has an $S^1$-symmetry, and is not rigid. In other words, the output is also a hat-type orbit $\hat \gamma$, and is moreover non-constant. One can now forget all asymptotic markers, and get an element of the moduli space computing the $L_\infty$ structure on the equivariant complex.
So, up to renaming the output $\hat \gamma$ to $\gamma$:
\begin{equation}
\label{eq:l_neq_to_eq}
l^k_{neq}(\hat \gamma_1,\ldots,\hat \gamma_k)=l^k(\gamma_1,\ldots,\gamma_k).
\end{equation}

\subsection{Gravitational descendants}
\label{subsec:grav_sh}
As above, let $M$ be a Liouville domain with $c_1(M)=0$; fix the data defining the Floer complex $CF^*_{S^1,+}(M)$.
As in Section~\ref{subsec:grav_inv}, fix a point $y\in M$ and assume that $J$ is integrable in its neighbourhood; fix  a germ $Y$ of a hypersurface defined in that neighbourhood.

For a domain $(\C P^1,z_0,\ldots,z_k)\in \M_{k+1}$, and consider the corresponding curve with punctures $$D=\C P^1\setminus\cup_{i=1}^k\{z_i\}$$ equipped with cylindrical ends. Note that the point $z_0$ is not being punctured.
For any $k\ge 1$, $m\ge 1$, and $x_i\in CF^*_{S^1,+}(M)$, consider the \emph{gravitational descendant moduli spaces}
	\begin{equation}
	\label{eq:m_gr_sh}
	\M_{1|k}(\psi_{m-1}\,\pt,x_1,\ldots,x_k)=\left\{
	\begin{array}{l}
	(D,u):D\in\M_{k+1},\ u\co D\to M,\\ 
	(du-X_H\otimes \beta)^{0,1}=0,\\
	u(z_0)=y ,\\
	u \text{ has local intersection} \\
	\quad \text{multiplicity }m\text{ with }Y\text{ at }u(z_0)=y,
	\\
	u\text{ is asymptotic to } x_i, \ i=1,\ldots,k.
	\end{array}
	\right\}
	\end{equation}
The equation in (\ref{eq:m_gr_sh}) carries a Hamiltonian perturbation, chosen exactly as in Section~\ref{subsec:psi_inv}.

The dimension of (\ref{eq:m_gr_sh}) equals
$$
2n+2(k+1)-6-2n-2(m-1)-\textstyle\sum_i|x_i|
$$
or equivalently
\begin{equation}
\label{eq:dim_m_gr_sh}
2(k-1)-2m-\textstyle\sum_i|x_i|.
\end{equation} 

Let $\C[-2m]$ be the 1-dimensional  graded vector space concentrated in degree $-2m$. In view of the above, for each $m,k\ge 1$ one has the gravitational descendant operations
\begin{equation}
\label{eq:tau_cf_dfn}
\psi_{m-1}^k\co CF^*_{S^1,+}(M)^{\otimes k}\to \C[-2m],\quad \deg\psi_{m-1}^k=2-2k.
\end{equation}
By definition of this being a graded map, it vanishes unless
(\ref{eq:dim_m_gr_sh}) holds, in which case it is set to count
$$
\psi_{m-1}^k(x_1,\ldots,x_k)\coloneqq\#\M_{1|k}(\psi_{m-1}\,\pt,x_1,\ldots,x_k).
$$
(The right hand side has been defined on the generators of $CF^*(M)$, and it is extended by linearity to their combinations.) It is clear from the definition that these operations are invariant under permutations of the inputs up to  sign, and orientation analysis similar to \cite{EO17} reveals that
$$
\psi_{m-1}^k(x_1,\ldots,x_k)=(-1)^\dagger\, \psi_{m-1}^k(x_{\sigma_1},\ldots,x_{\sigma_k})
$$
with
$
\dagger
$
is as above.

The curves in  higher-dimensional moduli spaces (\ref{eq:dim_m_gr_sh}) can undergo the following types of bubbling:
\begin{itemize}
	\item bubbling resulting from domain degenerations, when several marked points $z_i$ come together. The combinatorics of the  limiting holomorphic buildings is governed by the underlying compactification $\bM_{k+1}$;
	\item  bubbling of (augmented) Floer cylinders attached to inputs which happen independently of domain degenerations;
	\item the bubbling of a stable constant sphere at the tangency point $y$ which results in the splitting of a curve at $y$ similarly to the way described in the proof of Lemma~\ref{lem:M_codim_2}. Analogously to that lemma, this is a codimension~2 phenomenon.
\end{itemize}

For generic 1-dimensional descendant moduli spaces, only the first two types of bubbling happens, which leads to the next theorem.

\begin{theorem}
	\label{th:gr_l_inf}
For each $m\ge 1$, the collection of maps 
$$\psi_{m-1}=\{\psi_{m-1}^k\}_{k\ge 1}$$
defines an $L_\infty$ morphism from the $L_\infty$ algebra $CF^*_{S^1,+}(M)$ to the 1-dimensional vector space $\C[-2m]$ considered as the trivial $L_\infty$ algebra. In other words, the equations below hold. \qed
$$
\sum_{\stackrel{1\le r\le k,}{\sigma\in S_k}}(-1)^\maltese\, \tfrac{1}{r!(k-r)!}\, \psi^{k+1-r}_{m-1}(l^{r}(x_{\sigma_1},\ldots x_{\sigma_{r}}),x_{\sigma_{r+1}},\ldots, x_{\sigma_{k}})=0.
$$
\end{theorem}

\begin{remark}
	One can equivalently say that $\psi_{m-1}$ is a \emph{shifted $L_\infty$ augmentation} of $CF^*(M)$.
\end{remark}

The second type of bubbling is responsible for the appearance of the $l^1$-terms (the Floer differential) in Theorem~\ref{th:gr_l_inf}, and the first type of bubbling is responsible for the $l^r$, $r\ge 2$. The proof of Theorem~\ref{th:gr_l_inf} is entirely analogous to \cite{Sie19}, and can e.g.~also be obtained by a modification of~\cite{EO17}.

\begin{remark}
	\label{rmk:tan_aug}
	One can consider the purely holomorphic equation in (\ref{eq:m_gr_sh}). (In this case, when $k=1$, one takes quotient by the additional $\R_+$ symmetry.) This gives rise to an $L_\infty$ augmentation discussed by Siegel \cite{Sie19}, which is genuinely different from the one this paper considers. The reasons are analogous to the ones explained in Remarks~\ref{rmk:comp_1},~\ref{rmk:comp_2}.
	
	As noted by Siegel, the purely holomorphic version of the augmentation can also be understood from the perspective of cobordisms. By neck-stretching around a spherical neighbourhood  $S^{2n-1}=\del U(y)$ of $y$,  curves computing the purely holomorphic augmentations essentially reduce to  curves in the cobordism $M\setminus U(y)$ between $\del M$ and $S^{2n-1}$. In the case of the $\psi$-version of the augmentation, one could try stretching  around a bigger neighbourhood of $y$ which encloses the Hamiltonian support. The lower parts of the building are expected to become $\psi$-augmentation curves for the ball, which seem to be non-trivial.
\end{remark}

\subsection{Non-negatively graded domains}
In view of Theorem~\ref{th:gr_l_inf}, the descendant operations $\psi_{m-1}^k$ from (\ref{eq:tau_cf_dfn}) do not generally give rise to operations at the cohomology level.
On the other hand, recall that the domain  needed for proving Theorem~\ref{th:lg_per} is $M=T^*T^n$. In this case one can arrange that $CF^*_{S^1,+}(M)=SH^*_{S^1,+}(M)$ up to an arbitrarily high action truncation, so the operations $\psi_{m-1}^k$ do automatically become cohomological operations. The  definition below is a natural generalisation of this observation.

\begin{definition}
	\label{def:nonneg}
A Liouville domain $(M,\theta)$ is said to have \emph{non-negatively graded Floer complex} if for any $A>0$, there is a Liouville 1-form $\theta'$ on $M$ such that $\theta'-\theta=df$ and such that in the Floer complex $CF^*_{S^1,+}(M)$ defined using $\theta'$, every generator of action smaller than $A$ has non-negative degree. Equivalently, $\theta'|_{\del M}$ is a contact form all of whose Reeb orbits $\gamma$ of action smaller than $A$ have Conley-Zehnder index satisfying $\mu(\gamma)\le n-1$.
\end{definition}

\begin{example}
	\label{ex:non_pos_L}
The cotangent bundle of a manifold with a non-positive sectional curvature has non-negatively graded Floer complex \cite[Lemma~2.2]{CM14}. If the manifold has negative sectional curvature, one can find a $\theta$ making all orbits have non-negative symplectic cohomology degree, otherwise this is only  achieved up to an arbitrarily high action truncation.
\end{example}

\begin{example}
Let $X$ be a compact Fano variety and $\Sigma$ a smooth divisor in the class $Nc_1(X)$ for $N\ge 1$.	Then $X\setminus \Sigma$ has non-negatively graded Floer complex \cite{GP17}; in this case there is Liouville form $\theta$ all of whose Reeb orbits have non-negative symplectic cohomology degree.
\end{example}

Consider a Liouville manifold $M$ with a non-negatively graded Floer complex. For any $k\ge 2$ and generators 
$$x_i\in CF^0_{S^1,+}(M),$$
consider the counts
\begin{equation}
\label{eq:gr_sh_0}
\langle
x_1|\ldots|x_k
\rangle_M=\psi_{k-2}^k(x_1,\ldots,x_k)\in \Z.
\end{equation}

\begin{remark}
For any Liouville domain with non-negatively graded Floer complex, the degree~0 orbits $x_i$ are all of type $\hat{\gamma}$, not $\check{\gamma}$.
\end{remark}

In the rest of the paper, it will be assumed without further notice that all  generators appearing in (\ref{eq:gr_sh_0}) have action $<A$ for a sufficiently large $A$, and that a corresponding Liouville form $\theta'$ from Definition~\ref{def:nonneg} has been fixed.

\begin{proposition}
	\label{prop:gr_0_inv}
Consider a Liouville domain  $M$ with  non-negatively graded Floer complex. Then (\ref{eq:gr_sh_0}) descend to cohomological operations $SH^0_{S^1,+}(M)^{\otimes k}\to \Z$ invariant of the choices of $y$ and $Y$, and of compactly-supported homotopies of $J$.
\end{proposition}

\begin{proof}
Consider a moduli space (\ref{eq:m_gr_sh}) computing  (\ref{eq:gr_sh_0}) where the data $y$, $Y$ or $J$ vary in a 1-parametric family. The moduli space becomes 1-dimensional as well, and curves in this space can undergo the bubbling outlined above. 

The claim is that bubbling arising from domain degenerations does not happen. Suppose it happens; consider the part of the broken building which inherits the tangency condition. It belongs to a moduli space 
$$
\M_{1|k_0}(\psi_{m-1}\,\pt,x_1',\ldots,x_{k_0}')
$$
for some $k_0<k$ and $x_i'\in CF^*_{S^1,+}(M)$; according to Definition~\ref{def:nonneg} holds that $|x_i'|\ge 0$.
By (\ref{eq:dim_m_gr_sh}) the dimension of this moduli space is $\le -2$, hence such bubbling cannot happen. 

The only bubbling that can happen in codimension~1 is, therefore, the bubbling of  Floer cylinders from the input asymptotics, but they cancel when the inputs $x_i$ are Floer cocycles.
\end{proof}

\subsection{Descendants of the torus}
Recall that
$$
SH^0_{S^1,+}(T^*T^n)\cong \Z[H_1(T^n;\Z)]= \C[x_1^{\pm 1},\ldots,x_n^{\pm 1}]
$$
as vector spaces.
The following notation for the elements of $SH^0_{S^1,+}(T^*T^n)$ shall be adopted: for a vector $v_i\in H_1(T^n;\Z)\cong \Z^n$, the generator corresponding to it is denoted by 
$$
\xx^{v_i}=x_1^{v_i^1}\ldots x^{v_i^n}\in SH^0_{S^1,+}(T^*T^n).
$$
Here $v_i^j$ are the co-ordinates of $v_i\in\Z^n$.
The theorem below is the core computation of this paper; Section~\ref{sec:torus} is devoted to its proof.

\begin{theorem}
	\label{th:m_comput}
Suppose $v_1,\ldots,v_k\in H_1(T^n;\Z)$. Then
$$
\langle\xx^{v_1}|\ldots |\xx^{v_k}\rangle_{T^*T^n}=\begin{cases}
(k-2)!&\textit{if } v_1+\ldots+v_k=0,\\
0&\textit{otherwise}.
\end{cases}
$$
\end{theorem}

The left hand side is the descendant invariant (\ref{eq:gr_sh_0}).
It is clear that the descendant  is zero unless $v_1+\ldots+v_k=0$ because the curves computing it can only connect orbits whose sum is null-homologous.
A quick reflection persuades that the rest of the statememt is not obvious. Clearly the descendants from the theorem are invariant under the \emph{diagonal} action of $SL(n,\Z)$ on $(\Z^n)^k$, but this action has infinitely many orbits on $k$-tuples of vectors summing to zero.

\subsection{A reminder on symplectic cohomology operations}
This is a good occasion to recall some basic operations on the (non-equivariant) symplectic cohomology $SH^*(M)$; they will be used in Section~\ref{sec:torus}. 
They are the product
$$
-\cdot-\co SH^*(M)\otimes SH^*(M)\to SH^*(M),
$$
the bracket
$$
l^2_{neq}\co SH^*(M)\otimes SH^*(M)\to SH^{*-1}(M),
$$
and the BV operator
$$
\Delta\co SH^*(M)\to SH^{*-1}(M).
$$
The bracket $l^2_{neq}$ is part of the $L_\infty$ structure defined above, and descends to cohomology by the $L_\infty$ relations.
In the orientation framework being used for symplectic cohomology, both the product and the bracket are graded commutative: 
$$x\cdot y=(-1)^{|x||y|}\,y\cdot x,\quad  l^2_{neq}(x,y)=(-1)^{|x||y|}\, l^2_{neq}(y,x).$$
The three operations are related by the identity
\begin{equation}
\label{eq:delta_bracket}
l^2_{neq}(x,y)=\Delta(x\cdot y)-\Delta(x)\cdot y-(-1)^{|x|}x\cdot \Delta(y),
\end{equation}
see e.g.~\cite{Sei14,Sei16}. It is reminded that the BV operator acts at chain level by $\Delta(\check{\gamma})=\hat{\gamma}$ and $\Delta(\hat{\gamma})=0$; $\Delta$ also vanishes on the constant orbits.
\subsection{String topology}
Let $L$ be a smooth orientable spin $n$-manifold, and 
$M=T^*L$. 
The homology $H_* (\L L)$ of its free loop space, and more generally the space of chains $C_*(\L L)$ on the free loop space, carry an abundance of  algebraic structures. Their study goes under the general name of \emph{string topology} and was pioneered by Chas and Sullivan \cite{CS99}. Examples of cohomological operations include the Chas-Sullivan product
$$
-\cdot-\co H_*(\L L)\otimes H_*(\L L)\to H_{* -n}(\L L),
$$
the Chas-Sullivan bracket
$$
[-,-]\co H_*(\L L)\otimes H_*(\L L)\to H_{* -n+1}(\L L),
$$
and the BV operator
$$
\Delta\co H_{*}(\L L)\to H_{*+1}(\L L).
$$
These three are related by the identity analogous to (\ref{eq:delta_bracket}).

\begin{remark}
This paper uses the convention where the sign behaviour of the Chas-Sullivan bracket matches the symplectic cohomology one.
In \cite{CS99}, the graded-commutativity property of the bracket, and the signs in a version of (\ref{eq:delta_bracket}), are different. One brings them to match the symplectic cohomology signs by redefining $[x,y]\mapsto (-1)^{n+|x|}[x,y]$. \end{remark}

There exist other cohomological operations, see e.g.~\cite{God07,GH09}. And  importantly, there are vastly more operations defined at the chain level, i.e.~operations with inputs and outputs in $C_*(\L L)$. For example, the differential and the chain-level Chas-Sullivan bracket are expected to extend to a sequence of operations which turn $C_*(\L L)$ into an $L_\infty$ algebra, called the Chas-Sullivan algebra. This is sketched in \cite{CS99,Sul07}, but in general the definition of chain-level string topology operation meets a technical obstacle: the natural geometric definitions usually work only when the inputs are sufficiently \emph{transverse} chains of loops, and the issue lies in extending them to all chains. There are recently proposed solutions to this  issue \cite{Iri17, Iri18}.

Now, it is a fundamental fact that the Floer complex $CF^*(M)$ is a model for  the chains on the free loop space of $L$.
Indeed, the Viterbo theorem says that $SH^*(M)\cong H_{n-*}(\L L)$, and one expects there to be a complete correspondence between the algebraic structures on $CF^*(M)$ defined using Floer theory,  and the string topology operations on $C_{n-*}(\L L)$. 

For example, it is known that the Viterbo isomorphism intertwines the symplectic cohomology product and the BV operator on symplectic cohomology with the corresponding string topology appearing above.
The statement about the product is due to Abbondandolo and Schwarz \cite{ASc10},
and the monograph of Abouzaid \cite{Ab15} also includes the BV~structure.

\begin{theorem}
The Viterbo isomorphism $SH^*(T^*L)\cong H_{n-*}(\L L)$ is an isomorphism of BV algebras.\qed
\end{theorem}

The Lie brackets are recovered from the product and the BV operator as explained above, implying the next corollary which will be useful in Section~\ref{sec:pr}.

\begin{corollary}
\label{cor:vit_lie}
The Viterbo isomorphism intertwines the symplectic cohomology Lie bracket $l^2_{neq}$   with the Chas-Sullivan bracket.
\qed
\end{corollary}

The references given for the above theorem  use the Hamiltonian framework for symplectic cohomology but the continuation isomorphism \cite{BO09} between the SFT and Hamiltonian versions of symplectic cohomology is easily shown to intertwine the bracket, cf.~\cite{EO17}, so Corollary~\ref{cor:vit_lie} holds in the SFT framework as well.

\section{From Gromov-Witten to Landau-Ginzburg}
\label{sec:pr}
This section gathers the results from Sections~\ref{sec:cm},~\ref{sec:bs} and~\ref{sec:grav} together to prove Theorem~\ref{th:lg_per}. It is reminded that the discussion in Sections~\ref{sec:bs} and~\ref{sec:grav} is  more general than actually needed: it concerns general Liouville domains while the proof of Theorem~\ref{th:lg_per} only requires working with the domain $T^*T^n$.

\subsection{Proof of Theorem~\ref{th:lg_per}}
One picks up from the outcome from the end of Section~\ref{sec:cm}. Recall that  formula  (\ref{eq:def_psi_t}) introduces closed-string  enumerative descendants  $\langle\psi_{d-2}\,\pt\rangle_d^\en$,
computed (up to a normalising factor)
by
the curves in the moduli space (\ref{eq:m_psi}). Section~\ref{sec:cm} also explains how to stretch 
the curves (\ref{eq:m_psi})
 around a monotone Lagrangian torus $L\subset X$ to produce holomorphic buildings $(u,w_1,\ldots, w_d)$.
Recall that they consist of $u\subset T^*L$ and  holomorphic planes $w_1,\ldots,w_d\subset X\setminus L$.

Recall that before stretching, the curves had $Nd$ auxiliary marked points $\{z_i\}$ mapping to a Donaldson divisor $\Sigma$.
Each $w_i$ has intersection number $N$ with $\Sigma$, because $[\Sigma]$ is dual to $2N$ times the Maslov class in $H_2(X,L)$. 
After stretching, each $w_i$ inherits exactly $N$ points
among the
$\{z_i\}$. 
Forgetting these marked points, each $w_i$ becomes a curve computing the Borman-Sheridan class $\BS$, see Section~\ref{sec:bs}. 

Denote $M=T^*L$, embedded into $X$ as the Weinstein neighbourhood of $L$ which was used for stretching. The completion of $X\setminus M$ is naturally isomorphic to the completion of $X\setminus L$. 

Next consider the curve $u$.
Its domain has $d+1$ marked points: the point $z_0$ which carries the tangency condition, and $d$ input punctures asymptotic to $\gamma_i$.
Note that 
$$
| \gamma_i|=0,
$$
by rephrasing (\ref{eq:deg_cm_stretch}) and (\ref{eq:degree_dic}).Also recall that $u$ inherits an annulus with a Hamiltonian perturbation.
This way each $u$ becomes an element of the moduli space computing gravitational descendants (\ref{eq:gr_sh_0}).

The curves $w_i$ and $u$ are generically simple. The SFT compactness \cite{CompSFT03} and gluing \cite{Pa15} theorems imply that  up to a reordering of the inputs,  there is a sign-preserving bijection between the closed curves computing $\langle\psi_{d-2}\,\pt\rangle_d^\en$, and the buildings $(u,w_1,\ldots,w_d)$. One gets:
$$
\langle\psi_{d-2}\,\pt\rangle_d^\en=\frac 1 {(Nd)!}\cdot \frac{1}{d!}\cdot
\binom {Nd} {N,\ldots, N}\cdot (N!)^d
\cdot
\langle  \BS|\ldots|\BS\rangle_{T^*L}= \frac{1}{d!}\cdot\langle  \BS|\ldots|\BS\rangle_{T^*L}.
$$
Here the factor $\frac 1 {(Nd)!}$  comes from  (\ref{eq:def_psi_t}), $\frac{1}{d!}$ comes from reordering the inputs, the multinomial term comes from the ways of distributing the marked points among the $w_i$, and $(N!)^d$ from the divisor axiom applied to the $w_i$-curves.
\qed

\subsection{Generalisation to Liouville subdomains}

The theorem below generalises Theorem~\ref{th:lg_per} which is its special case when $M$ is taken to be the Weinstein neighbourhood of a monotone Lagrangian torus, see Example~\ref{ex:l_m_mon} and Example~\ref{ex:non_pos_L}.

\begin{theorem}
	\label{th:gw_bs}
	Let $X$ be a closed monotone symplectic manifold, and $M\subset X$ a monotone Liouville subdomain (Definition~\ref{def:m_monot}) admitting a non-negatively graded Floer complex (Definition~\ref{def:nonneg}).
	
	Additionally, assume that there exists a Donaldson divisor $\Sigma\subset X$ such that
	 $M\subset X \setminus\Sigma$ is exact, and $c_1^{rel}\in H^2(X,M)$ is proportional to the Poincar\'e dual of $\Sigma$ (notice that the cohomology is relative).
	
	Then
	$$
	\tfrac{1}{d!}\langle\underbrace{ \BS|\ldots|\BS}_d
	\rangle_M=\langle\psi_{d-2}\,\pt\rangle_{X,d}^\en
	$$
	where $\BS$ is the Borman-Sheridan class from Section~\ref{sec:bs}, the left-hand side descendants are from (\ref{eq:gr_sh_0}), and the right-hand side Gromov-Witten invariant is from (\ref{eq:def_psi_t}).
\end{theorem}	

\begin{proof} Choose $\Sigma$ from  (\ref{eq:def_psi_t}) to be the given one.
	The stretching argument from Section~\ref{sec:cm} may  be performed around $M$ rather than around a neighbourhood of $L$.
Once again the actions of the Reeb orbits $\gamma_i$ arising from the stretching are bounded by an apriori constant determined by the size of a Liouville collar of $M$ embeddable into $X$, and one arranges using Definition~\ref{def:nonneg} that the Conley-Zehnder indices of all such orbits satisfy $\mu(\gamma_i)\le n-1$.
 The argument in Section~\ref{sec:cm} only needs this assumption, and the assumption of monotonicity, to yield exactly the same structure of  broken buildings. (In the case when $\del M$ has contractible Reeb orbits, the curves $w_i$ are additionally augmented by holomorphic planes in $M$ in accordance with the definition of the Borman-Sheridan class.)	
 The rest of the proof is analogous to the proof of Theorem~\ref{th:lg_per} given in this section.
 \end{proof}
 
 \begin{remark}
  
  \begin{figure}[h]
  	\includegraphics[]{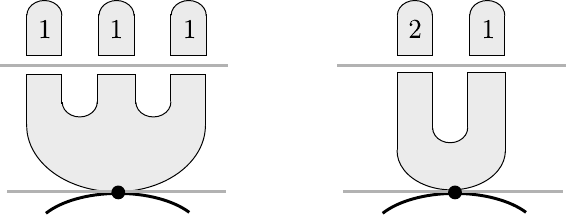}
  	\caption{Broken curves that may arise from stretching closed-string descendants when $CF^*(M)$ has negative degree generators. The numbers indicate $c_1^{rel}$.}
  	\label{fig:stretch_neg_degree}
  \end{figure}

As a final remark notice that the non-negatively graded assumption on $M$, or the non-positive sectional curvature assumption on $L$, was used crucially in the stretching argument to guarantee that Figure~\ref{fig:gw_stretch} is combinatorially the only possible configuration of broken curves. If $CF^*_{S^1,+}(M)$ has negative degree generators, different combinatorial types of broken curves become possible. For example, the stretching of $c_1=3$ spheres around $M\subset X$ may result in either of the broken curves shown in Figure~\ref{fig:stretch_neg_degree} (these two possibilities may not be exhaustive). This tallies with the fact that gravitational descendants of $M$ are no longer cohomological operations since they can undergo domain degenerations.
It would be interesting to obtain a version of Theorem~\ref{th:gw_bs} in this general case, involving chain-level  descendants.
\end{remark}

	\section{Gravitational descendants of the torus}
	\label{sec:torus}
This section presents a  proof of Theorem~\ref{th:m_comput}, the core computation of the paper. 	
	
\subsection{Basic relations}	
Consider the free loop space $\L T^n$ of the $n$-torus,
a vector $v\in \Z^n=H_1(T^n;\Z)$ and recall the notation for the corresponding symplectic cohomology class
$$
\xx^v=x_1^{v^1}\ldots x_n^{v^n}\in SH^0_{S^1,+}(T^*T^n)\cong \Z[H_1(T^n;\Z)]\cong\C[x_1^{\pm 1},\ldots,x_n^{\pm 1}].
$$
Here $v^i$ are the co-ordinates of $v$.
Now let $v_1,\ldots,v_k\in\Z^n$ be $k$ vectors; the aim is to compute descendant invariants
$$
\langle\xx^{v_1}|\ldots |\xx^{v_k}\rangle_{T^*T^n}\in\Z$$
from (\ref{eq:gr_sh_0}).
The subscript $T^*T^n$ will be dropped when not crucial.

The computation will rely on several basic properties for these invariants collected below, among which Proposition~\ref{prop:m_main_rel} is the most important one.

\begin{definition}
A collection of vectors $v_1,\ldots,v_k$ is said to be \emph{balanced} if $v_1+\ldots+v_k=0$.
\end{definition}

\begin{lemma}
	\label{lem:bal}
Descendants of the torus vanish on non-balanced inputs:
		\begin{equation}
		\label{eq:m_sum_0}
		\langle\xx^{v_1}|\ldots|\xx^{v_k}\rangle=0\quad\textit{unless}\quad
		\textstyle\sum_{i=1}^kv_i=0.
		\end{equation}
\end{lemma}
		
\begin{proof}
Any holomorphic curve contributing to a descendant invariant provides a nullhomology for the sum of its input Reeb orbits.
\end{proof}		
		
		\begin{lemma}
			\label{lem:SL}
			Descendants  of the torus are invariant under the action of $GL(n,\Z)$, that is, for any $g\in
			GL(n,\Z)$ one has
			$$
			\langle\xx^{v_1}|\ldots|\xx^{v_k}\rangle=\langle\xx^{g(v_1)}|\ldots|\xx^{g(v_k)}\rangle.
			$$
		\end{lemma}

\begin{proof}
	This is true because the $GL(n,\Z)$-action on the torus lifts to symplectomorphisms of its cotangent bundle.
\end{proof}

\begin{lemma}
\label{lem:stab} Descendants of the torus are invariant under stabilisation, that is, if $v_1,\ldots,v_k\in\Z^n\subset \Z^{n+1}$, then
$$
\langle\xx^{v_1}|\ldots|\xx^{v_k}\rangle_{T^*T^n}=\langle\xx^{v_1}|\ldots|\xx^{v_k}\rangle_{T^*T^{n+1}}.
$$
\end{lemma}
	
\begin{proof}
The usual technical issue with this type of statements is that a  product almost complex structure on $T^*T^n\times T^*S^1$ is not cylindrical, calling for a workaround. The one used here is inspired by \cite[Section~4]{DGI16}.

Let the tangency condition point $y\in T^*T^{n+1}$ defining gravitational descendants belong to the zero-section $T^{n+1}$.
Consider the splitting $T^{n+1}=T^n\times S^1$ and a perturbation of the flat metric on it with two totally geodesic tori of the form $T^n\times\{p\}$, $T^n\times \{P\}$, the first one containing geodesics of symplectic cohomology degree~0. All inputs $\xx^{v_i}$ correspond to closed geodesics within $T^n\times \{p\}$.
Assume that $y\in T^n\times \{p\}$. There is a compatible cylindrical almost complex structure $J$ on $T^*T^{n+1}$  such that $\Sigma=T^*(T^n\times\{p\})$ is a complex hypersurface. Pick a Hamiltonian vector field $X_H$ tangent to $\Sigma$.

Let $D\subset T^*T^{n+1}$ be the image of a curve $u$ contributing to the count of the descendant from the statement; it is enough to show that $D\subset \Sigma$ so assume that this is not the case.
Then for a generic such $J$, $D\cap\Sigma$ is discrete and non-empty because $y\in D\cap \Sigma$. So the intersection number $D\cap\Sigma$ is positive. This intersection number, by definition, only counts interior intersections and  ignores the fact that  $D$ is asymptotic to $\Sigma$ at infinity.

The projection $\pi\co T^*T^{n+1}\to T^*S^1$ to the last factor contracts the asymptotic orbits of $D$ to the point $p\subset T^*S^1$, and with this contraction the map $\pi|_D$ is null-homotopic because $T^*S^1$ is aspherical. So $\pi|_D$  has zero degree. The claim is that, on the other hand,  $D\cap \Sigma$ equals that degree, hence the contradiction. The claim holds because, since $\pi$ contracts the asymptotics of $D$ to $p$, the restriction of $\pi$ to a neighbourhood of each asymptotic of $D$ has a well-defined and vanishing degree, and  the fact that the asymptotics of $D$ are contained in $\Sigma$  may be ignored for the purpose of computing the total degree as the intersection  number $D\cap\Sigma$.
\end{proof}

\begin{lemma}
	\label{lem:cpm}	
Let $k=\sum_{i=1}^s (k_i+1)$, $n=\sum_{i=1}^s k_i$,	and $$\{x_{i,j}\}_{i=1,\ldots, s, \ j=1,\ldots, k_i}$$ be the variables corresponding to a basis of $H_1(T^n)$, grouped into $s$ groups. Then
$$
\langle x_{1,1}|\ldots|x_{1,k_1}|x_{1,1}^{-1}\ldots x_{1,k_1}^{-1}|x_{2,1}|\ldots |x_{2,k_2}|x_{2,1}^{-1}\ldots x_{2,k_2}^{-1}| \ldots \rangle_{T^*T^n}=(k-2)!
$$
Above, for each $i=1,\ldots, s$ there is a group of $(k_i+1)$ inputs as follows: $$x_{i,1},\ \ldots,\  x_{i,k_i},\ \prod_{j=1}^{k_i}x_{i,j}^{-1}.$$
In other words, Theorem~\ref{th:m_comput} holds for such inputs.
\end{lemma}
\begin{proof}
Let $X=\prod_{j=1}^s \C P^{k_i}$, and $\beta\in H_2(X)$ be the class of multidegree $(1,\ldots,1)$. This means, if $h_i\in H_2(\C P^{k_i})\subset H_2(X)$ are the line classes, that $\beta=\sum h_i$.
As a general observation, Gromov-Witten invariants can be updraded to take into account the homology classes of holomorphic curves. In particular, consider the GW invariant
$$
\langle \psi_{k-2}\, \pt\rangle_{X,\beta}
$$
defined as usual but only allowing curves in class $\beta$. Note that $c_1(\beta)=k$. Descendant GW invariants for toric manifolds are known, see e.g.~\cite[Corollary~C.2]{CCGK16}, and in particular:
$$
\langle \psi_{k-2}\, \pt\rangle_{X,\beta}=1.
$$

Next consider the monotone Lagrangian torus $L\subset  X$ which is the product  of Clifford tori in each factor. Its LG~potential is:
$$
W_L=\sum_{i=1}^s \left( x_{i,1}+\ldots+ x_{i,k_i}+\prod_{j=1}^{k_i}x_{i,j}^{-1}\right).
$$

One turns to Theorem~\ref{th:gw_bs},
and Lemma~\ref{lem:bs_w} which relates the Borman-Sheridan  class with the disk potential.
In the present case, the domain $M$ in Theorem~\ref{th:gw_bs} is a neighbourhood of $L$; observe that the inclusion $H_2(L)\to H_2(X)$ vanishes.
In this case Theorem~\ref{th:gw_bs}
obviously upgrades to an identity between descendant GW~invariants in a given curve class $\beta\in H_2(X)$, and those disk tuples whose classes $\beta_i\in H_2(X,L)$ satisfy $\beta=\sum \beta_i$, assuming that $\sum \del \beta_i=0\in H_1(L)$.

Recall that $W_L$ has one monomial summand for each holomorphic disk $\beta_i$ contributing to the potential of $L$; there are $k$ disks in total.
The only $k$-tuple of holomorphic disks $\{\beta_i\}$ satisfying $\sum \beta_i=\beta$ is the tuple which uses each holomorphic disk once, without repetitions.

A version of Theorem~\ref{th:gw_bs} for the class $\beta$ and Lemma~\ref{lem:bs_w} yield the following:
$$
\langle \psi_{k-2}\, \pt\rangle_{X,\beta}^\en =\frac 1 {k!}\cdot k!\cdot \langle x_{1,1}|\ldots|x_{1,k_1}|x_{1,1}^{-1}\ldots x_{1,k_1}^{-1}|x_{2,1}|\ldots \rangle_{T^*T^n},
$$
where the right hand side is the desired invariant. The $k!$ factor is the number of the reorderings of all $k$ monomials of $W$. 
Lemma~\ref{lem:desc} finishes the proof.
\end{proof}		

\begin{remark}
If one considers the purely holomorphic version of the $L_\infty$ augmentation from Remark~\ref{rmk:tan_aug},
the numbers in a version of Lemma~\ref{lem:cpm} become different. They are somewhat harder to compute and will be treated in a forthcoming paper. 	
\end{remark}
		
Denote $$\Lambda^2(\Z^n)=
\Lambda^2(H^1(T^n;\Z))= H^2(T^n;\Z)\cong H_{2n-2}(T^n;\Z);$$
this space can be seen as the space of skew-symmetric bilinear maps $\Z^n\otimes\Z^n\to\Z$.
Below is the last and most important property which will allow to compute descendants of the torus.

\begin{proposition}
	\label{prop:m_main_rel}
Consider a collection of vectors $v_1,\ldots,v_k\in\Z^n$ and any
$$
\Omega\in \Lambda^2(\Z^n),\quad u\in\Z^n.
$$
 The following relation holds, where $\Omega(u,v_i)\in\Z$ are the pairings.
\begin{equation}
\label{eq:m_main_rel}
\sum_{i=1}^k\Omega(u,v_i)\cdot  \langle \xx^{v_1}|\ldots| \xx^{v_{i-1}}|\xx^{v_i+u}|\xx^{v_{i+1}}|\ldots|\xx^{v_k}\rangle=0.
\end{equation} 
\end{proposition}		
	
\begin{remark}
The inputs in (\ref{eq:m_main_rel}) are simultaneously either all balanced or not. In the latter case the proposition is obvious. 
The balanced case in Proposition~\ref{prop:m_main_rel} is the one when $\sum_{i=1}^k v_i=-u$. In this case it is useful to note that Theorem~\ref{th:m_comput} implies Proposition~\ref{prop:m_main_rel}; indeed, assuming Theorem~\ref{th:m_comput}, (\ref{eq:m_main_rel})  rewrites as 
$$
\sum_{i=1}^k\Omega(u,v_i)\cdot  (k-2)!=(k-2)!\cdot \Omega(u,\textstyle \sum_{i=1}^k v_i)=(k-2)!\cdot \Omega(u,-u)=0.
$$
The goal will be  to prove the converse, also using the previous properties.
\end{remark}

\begin{example}
	\label{ex:comp}
Let 
$$v_1=(-2,-1),\quad v_2=(1,0),\quad v_3=(1,0).$$
Denoting the variables by $x,y$ instead of $x_1,x_2$ for convenience, the corresponding monomials $\xx^{v_i}$ are
$$
x^{-2}y^{-1},\ x,\ x.
$$
Take $u=(0,1)$ so that $\xx^u=y$, and take $\Omega=dy\wedge dx$.
The relation from Proposition~\ref{prop:m_main_rel} gives
$$
-2\, \langle x^{-2}|x|x\rangle+\langle x^{-2}y^{-1}|xy|x\rangle+\langle x^{-2}y^{-1}|x|xy\rangle=0
$$
The last two terms are equal to~1 by Lemmas~\ref{lem:cpm} and~\ref{lem:SL} because $x$, $xy$ are monomials corresponding to a basis in $\Z^2$. Therefore Proposition~\ref{prop:m_main_rel} proves that
$$
\langle x^{-2}|x|x\rangle=1.
$$
Observe that no two unputs here form a basis in $\Z^2$ so this is a new identity justifying Theorem~\ref{th:m_comput}, which does not follow from Lemma~\ref{lem:cpm} directly. 
\end{example}

Theorem~\ref{th:m_comput} will be proved by applying
Proposition~\ref{prop:m_main_rel} in a systematic way to reduce a given balanced input sequence $\xx^{v_1},\ldots,\xx^{v_k}$ to a combination of sequences of the basic form appearing in Lemma~\ref{lem:cpm}. Dimensional stabilisation will be used along the way.

The proof of Proposition~\ref{prop:m_main_rel} is a consequence of the $L_\infty$ relations from Theorem~\ref{th:gr_l_inf} together with a computation of the string Lie bracket on the torus.

\begin{remark}
It is possible to give a different proof not relying on the $L_\infty$ relations or string topology, but using the fact that descendants are well-behaved under Viterbo restriction maps, and using a version of the wall-crossing formula for toric mutations, cf.~\cite{Pa13,Pa14,PT17,To17}. 
\end{remark}
		
\subsection{Lie bracket and a string topology computation}
For this computation, it is convenient to start with non-equivariant symplectic cohomology. Recall that
$$
H_*(\L T^n)\cong H_*(T^n)\otimes\Z[x_1^{\pm 1},\ldots,x_n^{\pm 1}];
$$
this isomorphism is realised as follows. The part $\Z[x_1^{\pm 1},\ldots,x_n^{\pm 1}]$ keeps track of the homology class of the loops in a given cycle, and the part $H_*(T^n)$ is obtained by evaluating at the origin of each loop.
It is convenient to identify
$$
H_*(T^n)\cong\Lambda^{*}(H_1(T^n))\cong  \Lambda^{n-*}(H^1(T^n))
$$
where $\Lambda^*$ is the exterior algebra.
For $a\in \Lambda^{n-*}(H^1(T^n))$ and $u\in H_1(T^n)$, denote the corresponding element of $H_*(\L T^n)$ by
$$
a\otimes \xx^u \in H_*(\L T^n).
$$
Recall that the Chas-Sullivan product 
$$
-\cdot -\co H_*(\L T^n)\otimes  H_*(\L T^n)\to H_{*-n}(\L T^n)
$$
is defined as follows. Consider two cycles $x,y$ of parametrised loops on $L$. One evaluates the loops in the $x$-cycle and in the $y$-cycle at $0\in S^1$, and concatenates them pairwise over the fibre product of these evaluations. For the torus, the product is easily computed:
\begin{equation}
\label{eq:torus_prod}
(a\otimes \xx^u)\cdot (b \otimes \xx^v)=(a\wedge b)\otimes \xx^{u+v}.
\end{equation}
Next recall the BV operator
$$\Delta\co H_*(\L T^n)\to H_{*+1}(\L T^n).$$ It modifies cycles of loops by letting the parametrisations of loops sweep a once-around turn.
One computes in the adopted notation 
\begin{equation}
\label{eq:torus_d}
\Delta(a\otimes \xx^u)=\iota_ua\otimes \xx^u
\end{equation}
where $\iota_ua\in\Lambda^{|a|-1}(H^1(T^n))$ is the  interior product, or the result of insertion.

\begin{lemma}
	\label{lem:bracket_t2}
Consider a 2-form $\Omega\in \Lambda^2(H^1(T^n))$
and vectors $u,v\in H_1(T^n)$. 
For bracket on the non-equivariant Floer complex, it holds that
$$
[\Delta(\Omega\otimes \xx^u),\,\xx^v]=\Omega(u,v)\cdot \xx^{u+v}.
$$
\end{lemma}
\begin{proof}
	Recall that the Chas-Sullivan bracket is expressed in terms of the product and the BV~operator via (\ref{eq:delta_bracket}). Specifically,
	$$
[\Delta(\Omega\otimes \xx^u),\xx^v]=\Delta(\Delta(\Omega\otimes \xx^u)\cdot \xx^v)-\Delta(\Delta(\Omega\otimes \xx^u))\cdot \xx^v-\Delta(\Omega\otimes \xx^u)\cdot \Delta(\xx^v).
$$	
The second term vanishes because $\Delta^2=0$, and the third term vanishes because according to the notation being used, $\xx^v\in H_n(\L L)$ so the image of $\Delta$ lands in $H_{n+1}(\L L)=0$. Therefore
$$
[\Delta(\Omega\otimes \xx^u),\,\xx^v]=\Delta(\Delta(\Omega\otimes \xx^u)\cdot \xx^v).
$$
The computation is continued using (\ref{eq:torus_prod}) and (\ref{eq:torus_d}).
\begin{multline*}
\Delta(\Delta(\Omega\otimes \xx^u)\cdot \xx^v)=\Delta((\iota_u\Omega\otimes \xx^{u})\cdot \xx^v)=\Delta(\iota_u\Omega\otimes \xx^{u+v})
\\
=\Omega(u,u+v)\cdot \xx^{u+v}=\Omega(u,v)\cdot\xx^{u+v}.
\end{multline*}
The result is as claimed.
\end{proof}

\begin{remark}
Before the emergence of string topology,  Goldman \cite{Gol86} discovered a Lie bracket on the space of free loops on a surface. Let $\hat \pi$ be the set of free homotopy classes of oriented unparametrised loops on the genus $g$ surface $L=\Sigma_g$ (i.e.~the set of conjugacy classes of its fundamental group). The Goldman bracket is an operation
$$
[-,-]_G\co \Z[\hat \pi]\otimes \Z[\hat \pi]\to \Z[\hat \pi].
$$
For two transverse unparametrised oriented loops $\alpha$, $\beta$, one puts
$$
[\alpha,\beta]_G=\sum_{p\in\alpha\cap\beta}\epsilon(p)\,\alpha*_p\beta
$$
where $\epsilon(p)$ is the intersection sign and $\alpha*_p\beta$ is the concatenation  of the two loops at $p$. 
From the point of view of string topology,  the Goldman bracket is equivalent to the Chas-Sullivan bracket  on the $S^1$-equivariant loop space homology,
$$
[-,-]\co H_0(\L L/S^1, L)\otimes H_0(\L L/S^1,L)\to H_0^{S^1}(\L L/S^1,L),
$$
The Goldman bracket on the torus admits the following straightforward computation, see \cite{Tab16a,Tab16b,Cha10}.
Identify $\Z^2=H_1(T^2)$ with the set of free homotopy classes of loops on $T^2$; then for any $u,v\in \Z^2$, the Goldman bracket is given by
	$$
	[u,v]_G=(u^1v^2-u^2v^1)(u+v).
	$$
	Here $u^i$, $v^i$ are the co-ordinates of $u$ resp.~$v$. Lemma~\ref{lem:bracket_t2} is a higher-dimensional relative of that. In fact, Lemma~\ref{lem:bracket_t2} will now be recast in  the equivariant context. 
\end{remark}

Returning to the symplectic cohomology bracket on the $n$-torus,  its part
\begin{equation}
\label{eq:lie_1_0}
l_{neq}^2\co SH^1(T^*T^n)\otimes SH^0(T^*T^n)\to SH^{0}(T^*T^n).
\end{equation}
will be of particular interest. Using the Viterbo isomorphism and Corollary~\ref{cor:vit_lie}, Lemma~\ref{lem:bracket_t2} translates into
\begin{equation}
l_{neq}^2(\Delta (\Omega\otimes \xx^u),\,\xx^v)=\Omega(u,v)\cdot \xx^{u+v}
\end{equation}
where
$$
\Omega\otimes \xx^u\in SH^2(T^*T^n),\quad 
\Delta (\Omega\otimes \xx^u)\in SH^1(T^*T^n),\quad
\xx^v\in SH^0(T^*T^n). 
$$
Finally, one translates this computation to the equivariant case.
For this, observe that the elements $\Delta(\Omega\otimes \xx^u)$ and $\xx^v$ from Lemma~\ref{lem:bracket_t2} are realised as  hat-type generators on $CF^*(T^*T^n)$. Indeed, $\xx^v$ is a hat-generator by definition, and the first element has this property because it lies in the image of the BV operator. Relabelling these elements to the corresponding elements of $CH^*_{S^1,+}(T^*T^n)$, in view of (\ref{eq:l_neq_to_eq}) one gets the following identity for the equivariant bracket:

\begin{equation}
\label{eq:l2_comput}
l^2(i^{-1}(\Delta (\Omega\otimes \xx^u)),\,\xx^v)=\Omega(u,v)\cdot \xx^{u+v}.
\end{equation}
Here $i$ is the map $SH^*_{S^1,+}(T^*T^n)\to SH^*_{+}(T^*T^n)$ from the Gysin exact sequence, and the above argument has shown that $\Delta (\Omega\otimes \xx^u)$ lies in the image of $i$.

\begin{proof}[Proof of Proposition~\ref{prop:m_main_rel}]
This follows by combining (\ref{eq:l2_comput}) with the $L_\infty$~relations from Theorem~\ref{th:gr_l_inf}.
Recall that the Floer differential on $CF^*(T^*T^n)$ vanishes so there is a natural identification $CF^*(T^*T^n)=SH^*(T^*T^n)$.
Consider the sequence of elements in $CF^*_{S^1,+}(T^*T^n)$ below:
$$
i^{-1}(\Delta(\Omega\otimes \xx^u)),\ \xx^{v_1},\ \ldots,\ \xx^{v_k}.
$$
They have degrees $1,0,\ldots,0$ respectively. Apply the $L_\infty$ morphism equation from Theorem~\ref{th:gr_l_inf}.
The only non-trivial $l^r$ operations that can be applied to a subcollection of the above inputs are $l^2(i^{-1}(\Delta(\Omega\otimes \xx^u)),\,\xx^{v_i})$, since all other possible applications of an $l^r$ land in negative degree and  vanish.

So by Theorem~\ref{th:gr_l_inf}, 
$$
\sum_{i=1}^k  \tfrac{1}{2!(k-2)!} \langle (l^2(i^{-1}(\Delta(\Omega\otimes \xx^u)),\,\xx^{v_i})|\ldots|\xx^{v_{i-1}}|\xx^{v_{i+1}}|\ldots|\xx^{v_k}\rangle=0.
$$
The common factor may be dropped,
and the first term may be moved to the $i$th position. Finally (\ref{eq:l2_comput}) gives
\begin{multline*}
\sum_{i=1}^k  \langle \xx^{v_1}|\ldots|\xx^{v_{i-1}}|\,l^2(i^{-1}(\Delta(\Omega\otimes \xx^u)),\,\xx^{v_i})\,|\xx^{v_{i+1}}|\ldots|\xx^{v_k}\rangle
\\
=
\sum_{i=1}^k\Omega(u,v_i)\cdot  \langle \xx^{v_1}|\ldots| \xx^{v_{i-1}}|\xx^{v_i+u}|\xx^{v_{i+1}}|\ldots|\xx^{v_k}\rangle
\end{multline*}
which proves Proposition~\ref{prop:m_main_rel}.
\end{proof}
	
	\subsection{Proof of Theorem~\ref{th:m_comput}}
The following shorthand notation will be used:
$$\langle v_1|\ldots |v_k\rangle\coloneqq \langle \xx^{v_1}|\ldots|\xx^{v_k}\rangle_{T^*T^n}.$$
By the stabilisation property from Lemma~\ref{lem:stab}, one may forget the actual value of $n$. It is convenient to think of the inputs $v_1|\ldots |v_k$ as of the columnus of a matrix. Lemma~\ref{lem:stab} says that one can append one or several rows of zeroes to the matrix without changing the invariant:
$$
\langle v_1|\ldots |v_k\rangle=
\left\langle 
\begin{array}{c|c|c}
v_1&\ldots& v_k\\
0&\ldots&0
\end{array}
\right\rangle.
$$

Instead of appending a rows of zeroes, now consider appending the row $(1,0,\ldots,0)$:
$$
\begin{array}{c|c|c|c}
v_1&v_2&\ldots& v_k\\
1&0&\ldots&0
\end{array}.
$$
The next step is to apply Proposition~\ref{prop:m_main_rel} taking the columns of the above matrix as the vectors $v_i$ in the statement of Proposition~\ref{prop:m_main_rel}. It remains to specify which $\Omega$ and $u$ to use.

Take $u=(0,\ldots,0,-1)$ to be the basic vector corresponding to the last coordinate; call the corresponding variable $z_1$ for convenience.
Take the 2-form $\Omega$ to be $dz_1\wedge \alpha$ where $\alpha$ is a 1-form chosen generically so that $\alpha(v_1)\neq 0$.
Proposition~\ref{prop:m_main_rel}
yields the following:
\begin{equation}
\label{eq:proof_1}
\alpha(v_1)
\left\langle 
\begin{array}{c|c|c}
v_1&\ldots& v_k\\
0&\ldots&0
\end{array}
\right\rangle
+
\sum_{i=2}^k
\alpha(v_i)
\left\langle 
\begin{array}{c|c|c|c|c|c}
v_1&v_2&\ldots&v_i&\ldots& v_k\\
1&0&\ldots&-1&\ldots &0
\end{array}
\right\rangle.
\end{equation}

Suppose one has shown that
\begin{equation}
\label{eq:proof_2}
\left\langle 
\begin{array}{c|c|c|c|c|c}
	v_1&v_2&\ldots&v_i&\ldots& v_k\\
	1&0&\ldots&-1&\ldots &0
\end{array}
\right\rangle=(k-2)!
\end{equation}
for all $i\ge 2$.
Since $v_1+\ldots+v_k=0$, it  follows from (\ref{eq:proof_1}) that 
$$
\left\langle 
\begin{array}{c|c|c}
v_1&\ldots& v_k\\
0&\ldots&0
\end{array}
\right\rangle=(k-2)!
$$
as desired.

To compute the left hand side of (\ref{eq:proof_2}),
consider the linear transformation defined by $z\mapsto z\xx^{-v_1}$. By Lemma~\ref{lem:SL},
\begin{equation}
\label{eq:proof_3}
\left\langle 
\begin{array}{c|c|c|c|c|c}
v_1&v_2&\ldots&v_i&\ldots& v_k\\
1&0&\ldots&-1&\ldots &0
\end{array}
\right\rangle
=
\left\langle 
\begin{array}{c|c|c|c|c|c}
0&v_2&\ldots&v_i+v_1&\ldots& v_k\\
1&0&\ldots&-1&\ldots &0
\end{array}
\right\rangle.
\end{equation}
One may redenote $v_i+v_1$ simply by $v_i$, and reorder the  inputs so that $v_i$ becomes $v_2$. The task is now to show that
\begin{equation}
\label{eq:proof_4}
\left\langle 
\begin{array}{c|c|c|c|c}
0&v_2&v_3&\ldots& v_k\\
1&-1&0&\ldots &0
\end{array}
\right\rangle=(k-2)!
\end{equation}
for all vectors $\{v_i\}_{i= 2}^n$ such that $\sum_{i= 2}^n v_i=0$. Here $v_i\neq 0$ for $i\ge 3$, but $v_2$ may be zero.

There are two possible cases: $v_2=0$ and $v_2\neq 0$. Consider the case $v_2\neq 0$ first.
One repeats a version of the above procedure, starting with $v_2$. One adds an extra row $(0,1,0,\ldots)$  to the matrix:
$$
\begin{array}{c|c|c|c|c}
0&v_2&v_3&\ldots& v_k\\
1&-1&0&\ldots&0\\
0&1&0&\ldots&0
\end{array}.
$$
Again apply Proposition~\ref{prop:m_main_rel} taking the columns as the vectors $v_i$ in the Proposition.
Take $u=(0,\ldots,0,-1)$ to be the basic vector corresponding to the (new) last coordinate; call the corresponding variable $z_2$.
Take the 2-form $\Omega$ to be $dz_2\wedge \alpha$ where $\alpha$ is a 1-form such that $\alpha(v_2)\neq 0$. The fact that $v_2\neq 0$ is being used here.

Proposition~\ref{prop:m_main_rel}
yields the following:
\begin{equation}
\label{eq:proof_5}
\alpha(v_2)
\left\langle 
\begin{array}{c|c|c|c|c}
0&v_2&v_3&\ldots& v_k\\
1&-1&0&\ldots&0\\
0&0&0&\ldots&0
\end{array}
\right\rangle
+
\sum_{i=3}^k
\alpha(v_i)
\left\langle 
\begin{array}{c|c|c|c|c|c|c}
0&v_2&v_3&\ldots& v_i&\ldots& v_k\\
1&-1&0&\ldots&0&\ldots \\
0&1&0&\ldots&-1&\ldots& 0
\end{array}
\right\rangle.
\end{equation}
The first summand is the one appearing in (\ref{eq:proof_4}).
Using the fact that $\alpha(v_2)\neq 0$, to establish~(\ref{eq:proof_4})
it suffices to treat the other summands and prove:
\begin{equation}
\label{eq:proof_6}
\left\langle 
\begin{array}{c|c|c|c|c|c|c}
0&v_2&v_3&\ldots& v_i&\ldots& v_k\\
1&-1&0&\ldots&0&\ldots \\
0&1&0&\ldots&-1&\ldots& 0
\end{array}
\right\rangle
=(k-2)!
\end{equation}
for all $i\ge 3$. Now consider the linear change of coordinates given by $z_2\mapsto z_2z_1\xx^{-v_2}$.
This transforms (\ref{eq:proof_6})  to
\begin{equation}
\label{eq:proof_7}
\left\langle 
\begin{array}{c|c|c|c|c|c|c}
0&0&v_3&\ldots& v_i+v_2&\ldots& v_k\\
1&0&0&\ldots&-1&\ldots \\
0&1&0&\ldots&-1&\ldots& 0
\end{array}
\right\rangle
\end{equation}
Again, redenoting $v_i+v_2$ to $v_i$ and reordering so that $v_i$ becomes $v_3$ reduces the task to showing that
\begin{equation}
\label{eq:proof_8}
\left\langle 
\begin{array}{c|c|c|c|c|c}
0&0&v_3&v_4&\ldots & v_k\\
1&0&-1&0&\ldots&0 \\
0&1&-1&0&\ldots& 0
\end{array}
\right\rangle=(k-2)!
\end{equation}
for all vectors $\{v_i\}_{i= 3}^n$ such that $\sum_{i= 3}^n v_i=0$. Here $v_i\neq 0$ for $i\ge 4$, but $v_3$ may be zero.
The pattern is clear: if $v_3$ is non-zero, one continues analogously. For example, after one more iteration one reduces to descendants of the form
$$
\left\langle 
\begin{array}{c|c|c|c|c|c|c}
	0&0&0&v_4&v_5&\ldots & v_k\\
	1&0&0&-1&0&\ldots&0 \\
	0&1&0&-1&0&\ldots& 0\\
	0&0&1&-1&0&\ldots &0
\end{array}
\right\rangle
$$

On the other hand if, at one of these steps the vector $v_2$, or $v_3$, or $v_4$ etc.~happens to be zero, one simply ignores its column, moves one column right, and continues the procedure. For example, assuming $v_4=0$ in the matrix just above,
$$
\begin{array}{c|c|c|c|c|c|c}
0&0&0&0&v_5&\ldots & v_k\\
1&0&0&-1&0&\ldots&0 \\
0&1&0&-1&0&\ldots& 0\\
0&0&1&-1&0&\ldots &0
\end{array}
$$
one repeats the procedure starting from the column of $v_5$.

Iterating this procedure reduces the computation to the basic inputs which are the columns of a matrix of the following form, where all empty entries are zero, and the identically zero rows are omitted:
$$
\begin{array}{|c c c c c| c c c c c| c c c c c c }
\hline
1&&\ldots&&-1&&&&&&& \\
&1&\ldots&&-1&&&&&&& \\ 
&&&&&&&&&&& \\
&&\ldots&1 &-1 &&&&&&\\
\hline
&&&& &1&&\ldots&&-1&& \\
&&&&&&1&\ldots&&-1&& \\ 
&&&&&&&&&&&\\
&&&&&&&\ldots&1 &-1 & &\\
\hline
&&&&&&&& & & 1&\ldots
\end{array}
$$
These are precisely the inputs appearing in Lemma~\ref{lem:cpm}, and  Theorem~\ref{th:m_comput} follows.
	\bibliography{Symp_bib}{}
	\bibliographystyle{plain}

\end{document}